\documentclass[12pt]{amsart}
\usepackage{amsmath,amsopn,amsfonts,amssymb,amsthm}
\usepackage{graphicx}
\usepackage{float}

\theoremstyle{definition}

\newtheorem{examp}{Example}
\newtheorem{remark}{Remark}

\theoremstyle{theorem}
\newtheorem{thm}{Theorem}
\newtheorem{lem}{Lemma}
\newtheorem{coro}{Corollary}

\DeclareMathOperator{\CC}{\mathbb C}
\DeclareMathOperator{\RR}{\mathbb R}
\DeclareMathOperator{\ZZ}{\mathbb Z}
\DeclareMathOperator{\NN}{\mathbb N}
\DeclareMathOperator{\TT}{\mathbb T}
\DeclareMathOperator{\DD}{\mathbb D}
\DeclareMathOperator{\supp}{supp}

\DeclareMathOperator{\capa}{cap}
\DeclareMathOperator{\men}{men}

\newcommand{\too}{\longrightarrow}

\begin{document}

\title{The Hessenberg matrix and the Riemann mapping function}
\author{C.~Escribano}\email{cescribano@fi.upm.es}
\author{A.~Giraldo}\email{agiraldo@fi.upm.es}
\author{M.~A.~Sastre}\email{masastre@fi.upm.es}
\author{E.~Torrano}\email{emilio@fi.upm.es}
\address{Departamento de Matem\'{a}tica Aplicada, Facultad de Inform\'{a}tica, Campus de Montegancedo,
Universidad Polit\'{e}cnica de Madrid, 28660 Boadilla del Monte, Spain}

\begin{abstract}

We consider a Jordan arc $\Gamma $ in the complex plane $\CC $ and a regular measure $\mu $ whose support is  $\Gamma $. We denote by $D$ the upper Hessenberg matrix of the multiplication by $z$ operator with respect to the orthonormal polynomial basis associated with $\mu $. We show in this work
that,  if the  Hessenberg matrix $D$ is uniformly asymptotically
Toeplitz, then the symbol of the limit operator is the restriction to the unit circle of the Riemann mapping function $\phi(z)$ which maps conformally the exterior of the unit disk onto the exterior of the support of the measure $\mu $.

We use this result to show how to approximate the Riemann mapping function for the support of $\mu $ from the entries of the Hessenberg matrix $D$.

\end{abstract}

\maketitle

\section{Introduction}

In this paper, we consider regular measures $\mu $ defined on subsets of the complex plane that are Jordan arcs, or finite union of Jordan arcs such that its complement is simply connected, and we show how the entries of the Hessenberg matrix $D$ associated with $\mu $ determines the Riemann mapping function that takes the complement of the closed unit disk $\overline{\DD }$ to the complement of the support of $\mu $.

The Riemann mapping theorem says, in its most common statement (see, for example, \cite{ahlfors}), that given a simply connected domain $\Omega  \subsetneq \CC $ and given $z_0 \in \Omega$, there is a unique analytic map  $\phi:  \Omega  \longrightarrow \DD $ ($\DD $ the open unit disk), normalized by the conditions $\phi (z_0)=0 $ and $\phi '(z_0 )>0$, such that $\phi $ defines a one-to-one mapping of $\Omega $ onto $\DD $.
However, we will use an equivalent formulation for domains containing $\infty $ which can be found, for example, in \cite{szego, pommerenke, jakimovski}. In this case, the Riemann mapping theorem states that, for every $\Gamma \subset \CC $ compact that is not a point, such that $\CC_{\infty}  \setminus \Gamma$ is simply connected, there is a unique conformal mapping  $\phi: \mathbb{C}_{\infty} \setminus \overline{\DD }\rightarrow \CC_{\infty}  \setminus \Gamma$ such that $\phi (\infty )=\infty $ and
$\phi ' (\infty )>0$, where $\phi ' (\infty )=\capa (\Gamma) $, the capacity of $\Gamma $.
Moreover, if $\Gamma $ is a simple Jordan curve,  $\phi (z)$ is continuous in the unit circle $\TT$.

There exists a well-known  link between the Riemann mapping function and the Green function, which has been described in the literature on potential theory (see, for example, \cite{finkelshtein}). If we denote by $\Phi (z)$ the inverse of $\phi (z)$ as defined in the previous paragraph, then the Green function for a compact set $K$ with $\capa (K)>0$, $g_K (z,\infty )$,
can be obtained from the Riemann conformal mapping $\Phi (z)$ which takes $\CC _{\infty } \setminus P_c (K)$ onto the exterior of $\overline{\DD } $, where $P_c (K)$ is the polynomial convex hull of $K$. Moreover, if $\Phi (z)=\displaystyle{\sum _{k=-1}^{\infty } c_{-k} z^{-k } } $, $c_1 >0 $, in a neighborhood of $\infty $, then
$$
g_{K} (z,\infty )=
\left\{ \begin{array}{ll}
\log |\Phi (z)| & \mbox{if  }\; z\in \CC _{\infty } \setminus P_c ( K) \\
0 & \text{otherwise.}
\end{array} \right.
$$
In particular, $c_1 =\dfrac{1}{\capa (K)} $.

An infinite matrix $T=(a_{i,j})_{i,j=1}^{\infty}$ is a Toeplitz matrix if each descending diagonal from left to right is constant, i.e,
there exists $(a_i)_{i\in \ZZ }$ such that
$$T= \left (
\begin{array}{ccccc}
a_{0}& a_{-1} & a_{-2} & \ldots \\
a_{1} & a_{0} & a_{-1} & \ldots \\
a_{2} & a_{1} & a_{0} & \ldots \\
a_{3} & a_{2} & a_{1} & \ldots \\
\vdots & \vdots & \vdots & \ddots
\end{array}
\right ).$$
Given a Toeplitz matrix $T$, the Laurent series whose coefficients are the entries $a_i $ defines a function known as the symbol of $T$.

J.Barr\'{\i}a and P.R.Halmos
\cite{hal-barria} called an operator $A$ on $\ell ^2$ asymptotically Toeplitz if
the sequence of operators $\{S^{*n}AS^n\}$ converges strongly  to a bounded operator on $\ell ^2$,
where $S$ is the forward shift and $S^*$ its adjoint.
  Later, A.Feintuch
  \cite{feintuch} established the definitions of weak and uniform operator convergence.
  Hence there are actually three different kinds
of asymptotic toeplitzness. In this work we deal with the uniform one: a bounded operator $A$  on  $\ell^2$  is uniformly asymptotically Toeplitz provided there is
a bounded operator $T$ on $\ell ^2$ such that
   $$ \lim_{n\to \infty } \| S^{*n} A S^n-T\| =0 .$$
By Theorem 2.4 in \cite{feintuch}, uniformly asymptotically Toeplitz operators are characterized as the compact perturbations of Toeplitz operators.

In the real case, the Hessenberg matrix agrees with the tridia\-gonal Jacobi matrix. Rakhmanov's theorem \cite{rakh77,rakh83} states that, if the support of a Borel measure is $[-1,1]$ with $\mu '> 0$ almost everywhere in $[-1,1]$, and
$$J= \left (
\begin{array}{ccccc}
b_{0}& a_{1} & 0 & \ldots \\
a_{1} & b_{1} & a_{2} & \ldots \\
0 & a_{2} & b_{2} & \ldots \\
0 & 0 & a_{3} & \ldots \\
\vdots & \vdots & \vdots & \ddots
\end{array}
\right ) $$
is the Jacobi matrix associated with $\mu $, then $a_n \rightarrow \dfrac{1}{2} $ and $b_n \rightarrow 0$. We observe that this matrix operator  is a compact perturbation of a Toeplitz
operator, so that $J$ is uniformly asymptotic Toeplitz.

Note that in this case, the Riemann mapping function for the interval $[-1,1]$ is
 $$\phi(z)=\frac{1}{2}\left(z+\frac{1}{z}\right)=\frac{1}{2}z+0+\frac{1}{2} \frac{1}{z}$$

Moreover, under the conditions of Rakhmanov's theorem, if $\supp (\mu )=[a,b]$, then the above limits are $a_{n} \to \dfrac{b-a}{4} $ and $b_{n} \to \dfrac{a+b}{2}$. In this case, the Riemann mapping function is
 $$\phi(z)=\frac{b-a}{4}z+ \frac{a+b}{2}+\frac{b-a}{4} \frac{1}{z}.$$

Conversely, P.Nevai established in \cite{nev79} that, if $a_{n} \to a>0$ and $b_{n} \to b$, then the support is $[-2a+b,b+2a]\cup e$ (where $e$ is at most a denumerable set of isolated points). Moreover, Nevai proved the equivalence between the existence of those limits and the ratio asymptotics
of orthonormal polynomials.

Generalizations of Rakhmanov's theorem to orthogonal polynomials, and to orthogonal matrix polynomials on the unit circle, have been given in \cite{mate85} and \cite{vanasche}. The case of orthogonal polynomials on an arc of circle has been studied in \cite{bello}.

As a final introductory motivating fact, we consider the Hessenberg matrix $D$ associated with a measure $\mu $ on $\TT $,
$$ D= \left (
\begin{array}{ccccc}
d_{1,1}& d_{1,2} & d_{1,3} & \ldots \\
d_{2,1} & d_{2,2} & d_{2,3} & \ldots \\
0 & d_{3,2} & d_{3,3} & \ldots \\
0 & 0 & d_{4,3} & \ldots \\
\vdots & \vdots & \vdots & \ddots
\end{array}
\right ) .$$
Then, if $\mu $ is of Szeg\H{o} class \cite{geronimus}, $ \lim _{n\to \infty }d_{n+1,n} = 1$ and $\lim _{n\to \infty }d_{n-k,n} = 0$, for all $k=0,1,2,\dots $ On the other hand, $\phi (z)= z $.

In  this paper we prove a theorem relating the limits of the elements of the diagonals of the Hessenberg matrix of a measure $\mu $, with the coefficients of the Laurent series expansion of the Riemann mapping function for the support of $\mu $. Specifically, we  show that, if $\mu $ is a regular measure in the complex plane $\CC $, whose support is  a Jordan arc or a finite union of Jordan arcs such that its complement is simply connected, and  if the  Hessenberg matrix operator $D$ is uniformly asymptotically
Toeplitz, then the symbol (whenever it is continuous) of the limit operator is the restriction to the unit circle of the the Riemann mapping function $\phi(z)$ which maps conformally the exterior of the unit disk onto the exterior of the support of the measure $\mu $. This connection is certainly a ``natural one'' for Faber polynomials \cite{gaier} and leads to a method for approximating the conformal map.

As a consequence of this theorem, we give a method to approximate the Riemann mapping function from the entries of the Hessenberg matrix $D$ even when the Hessenberg matrix $D$ can not be obtained as a closed form, and it is not possible to compute the limits of the elements of its diagonals.
 Also,  it is still possible to compute approximations of the support of the measure $\mu $, computing the image of the unit circle under suitable approximations of the Riemann mapping function.

There exist some previous results relating the properties of $D$ and the support of $\mu $. For example, if the Hessenberg matrix $D$ defines a subnormal operator \cite{halmos2} on $\ell^2$, then the closure of the convex hull of its numerical range agrees with the convex hull of its spectrum. On the other hand, the spectrum of the matrix $D$ contains the spectrum of its minimal normal extension $N=\men(D)$ which is precisely the support of the measure \cite{conway2}.

The organization of the paper is as follows: In Section $2$ we prove the main theorem.
 In  Section $3$ we show that the Riemann mapping function for the support of $\mu $ can be approximated from the entries of the Hessenberg matrix $D$. The last Section is devoted to several heuristic examples to illustrate the approximation results given in previous section.

For general information on the theory of orthogonal polynomials, we recommend the books \cite{chihara, szego} by T. S. Chihara and G. Szeg\H{o}, respectively, and the survey \cite{golinskii-totik} by L. Golinskii and V. Totik.

%%%%%%%%%%%%%%%%%%%%%%%%%%%%%%%%%%%%%%%%%%%%%%%%%%%%%%%%%%%%%%%%%%%%%%%%%%%%%%%%%%%%%%

\section{The Diagonals Theorem}

Let $\mu$  be a Borel probability measure in the complex plane, with support $\supp(\mu)$ containing  infinitely many points.
Let $\mathcal{P} $ be the space of polynomials. The associated inner product is given by the expression
\[\langle
Q(z), R(z) \rangle _{\mu }= \int _{\supp (\mu )} Q(z) \overline{R(z)} d\mu(z),\]
for $Q,R\in \mathcal{P} $.
Then there exists a unique orthonormal polynomials sequence (ONPS)
$\{P_{n}(z)\}_{n=0}^{\infty}$ associated to the measure $\mu $ \cite{chihara,
fre61, szego}.

In the space $\mathcal{P} ^2 (\mu )$, closure of the polynomials space $\mathcal{P} $ in $L^2_{\mu } (\Omega )$, we consider the multiplication by $z$ operator. Let $D=(d_{ij})_{i,j=1}^{\infty}$ be the infinite upper Hessenberg  matrix of this operator in the basis of ONPS $\{ P_n(z) \}_{n=0}^{\infty } $, hence
\begin{equation} \label{rl0}
z P_{n} (z) = \sum_{k=0}^{n+1} d_{k+1,n+1} P_{k} (z), \quad n \geq 0,
\end{equation}
with
$P_{0}(z)=1$.

It is a well-known fact that the monic polynomials are the characteristic polynomials of the finite sections of $D$.

To state our main result, we will require the measure $\mu $ to be regular with support a finite union of Jordan arcs such that its complement is simply connected, and we will also need to consider an auxiliary Toeplitz matrix. We next recall the definitions of all these notions.

A Jordan arc in $\CC $ is any subset of $\CC $ homeomorphic to the closed interval $[0,1]$ on the real line.

A measure $\mu$ is regular if $\displaystyle{\lim_{n \to \infty} \frac{1}{\sqrt[n]{\gamma_{n}}} = \capa(\supp(\mu))}$, the capacity of the support of $\mu $, where the $\gamma _n $ are the leading coefficients of the orthonormal polynomials, i.e., $P_{n}(z)=\gamma_{n} z^{n}+ \cdots $.

We are now in a position to state and prove the main result of the paper.

\begin{thm}[The diagonals theorem]\label{diagonals th}
Let  $D=(d_{ij})_{i,j=1}^{\infty}$ be the Hessenberg matrix associated with a measure $\mu $ with compact support on the complex plane.
 Assume that:
\begin{enumerate}
\item The measure $\mu$ is regular with support $\supp(\mu) $ a Jordan arc or a finite union of Jordan arcs $\Gamma$ such that $\CC_\infty \setminus \Gamma $ is a simply connected set of the Riemann sphere $\CC_{\infty} $.
\item The Hessenberg matrix operator $D$ is uniformly asymptotically Toeplitz and the corresponding  limit matrix $T$   has a continuous symbol.
\end{enumerate}
Then, the symbol of $T$ is the restriction to the unit circle of the the Riemann mapping function  $\phi:\CC_{\infty} \setminus \overline{\DD }\rightarrow \CC_{\infty}  \setminus \Gamma$ which maps conformally the exterior of the unit disk onto the exterior of the support of the measure $\mu $.

\end{thm}

\begin{proof}
Suppose that
$$T= \left (
\begin{array}{ccccc}
d_{0}& d_{-1} & d_{-2} & \ldots \\
d_{1} & d_{0} & d_{-1} & \ldots \\
0 & d_{1} & d_{0} & \ldots \\
0 & 0 & d_{1} & \ldots \\
\vdots & \vdots & \vdots & \ddots
\end{array}
\right ).$$
Hence the symbol of $T$
$f_T:\TT \to \CC$ is given by
$$f_{\tiny { T}}(z)=d_1 z+ d_0 +d_{-1} \dfrac{1}{z} +d_{-2} \dfrac{1}{z^2}+ \cdots =\sum _{k=-1}^{\infty } d_{-k} {z^{-k}}.$$
Consider the formal series
$$ d(z)= \sum _{k=-1}^{\infty } d_{-k} {z^{-k}} \text{ with } \: \: |z|>1,$$
with defines a continuous function on $\CC\setminus \DD $.

On the other hand, by the Riemann mapping theorem \cite{pommerenke, jakimovski}, given $\Gamma \subset \CC $ compact, since $\CC_{\infty}  \setminus \Gamma$ is simply connected, there is a unique conformal mapping  $\phi: \mathbb{C}_{\infty} \setminus \overline{\DD }\rightarrow \CC_{\infty}  \setminus \Gamma$, with the expression
$$\phi (z)=c_1 z+ c_0 +c_{-1} \dfrac{1}{z} +c_{-2} \dfrac{1}{z^2}+ \cdots =\sum _{k=-1}^{\infty } c_{-k} {z^{-k}},$$
such that $c_1 >0$, with $c_1 =\capa (\Gamma) $ the capacity of $\Gamma $.

In order to prove the theorem it suffices to show that $d(z)$ satisfies the properties that determine the Riemann mapping function.

We will prove the following:
\begin{enumerate}
\item [(i)] $d(\TT)=\Gamma $.
\item [(ii)] The function $d(z)$ is analytic in $\CC \setminus \overline{\DD }$.
\item [(iii)]The function  $d(z)$ is univalent in $\mathbb{C}_{\infty} \setminus \overline{\DD }$.
\item [(iv)]The first coefficient of $d(z)$ is the capacity of $\Gamma $, i. e. $d_1=\capa (\Gamma )$.
\end{enumerate}

\noindent (i)
Observe first that $ \sigma_{ess}(T)=f_T(\TT )= d(\TT) $, since the symbol is a continuous function \cite{bottcher}.

On the other hand, Feintuch  \cite{feintuch}  characterized the uniformly asymptotically Toeplitz operators as just the compact perturbations of the Toeplitz
operators. Then since the essential spectrum (see, for example, \cite{conway2} for a definition) is invariant via compact perturbations \cite{conway}, we have that $\sigma_{ess}(D)= \sigma_{ess}(T)$.

On the other hand, since $\supp(\mu)=\Gamma$ is a compact set with empty interior and $\CC _{\infty }\setminus \Gamma $ is connected, we can apply Merguelyan's theorem (or Walsh' s theorem for a Jordan arc)  \cite[p.97]{gaier} to deduce that every continuous function in $\Gamma$ can be uniformly approximated by polynomials.

Since the set of continuous functions
 with compact support
is dense in $L^{2}_{\mu}(\Gamma)$, then $L^{2}_{\mu}(\Gamma)=P^{2}_{\mu}(\Gamma)$. Therefore, $D$ defines a normal operator on $\ell ^2$, hence $\sigma (D)=\Gamma $, see \cite[p. 41]{conway2}.
Since the support has no isolated points we have $\sigma(D)=\sigma_{ess}(D)  $.

Therefore,
$$d(\TT )=\sigma_{ess} (T)=\sigma_{ess} (D)=\sigma (D)=\Gamma .$$

\medskip

\noindent (ii)
To show that $d(z)=d_1 z+ \displaystyle{\sum _{k=0}^{\infty } d_{-k} z^{-k}}$ is  analytic in $\CC  \setminus \overline{\DD } $, we have, on the one hand, that the function given by the first summand $d_1 z $ is analytic in $\CC  $.
On the other hand, consider $\widetilde{d}(z)= \displaystyle{\sum _{k=0}^{\infty } d_{-k} z^{-k}}$.

If $|z|>1$, applying Chauchy-Schwarz inequality we have
$$\sum_{k=0}^{\infty}  \left|\dfrac{d_{-k} }{z^{k}}\right| \leq  \sqrt{\sum_{k=0}^{\infty} |d_{-k}|^{2}} \;
\sqrt{\sum_{k=0}^{\infty}|z|^{-2k}} \leq ||\widetilde{d} ||_2 \sqrt{\frac{1}{1-|z|^2}}<\infty .$$

Therefore, $\widetilde{d}(z)$ is
 analytic in $\CC  \setminus \overline{\DD } $.

\medskip

\noindent (iii)
Now, we have that $d(z)$ is  analytic in $\CC \setminus \overline{\DD } $ and it is continuous in $\CC \setminus \DD $. Moreover,  the set of limit points of $d(z)$ as $|z|\rightarrow 1 $ agrees with $d(\TT )=\Gamma $ which is bounded, without interior points and does not disconnect $ \CC _{\infty } $. Therefore  $d(z)$ is a conformal map in $\CC_{\infty} \setminus \overline{\DD } $ (see
\cite[Th 1.1]{pommerenke}).

\medskip

\noindent (iv) We finally show that $d_1=\capa (\Gamma )$.

The elements $d_{n+1,n}$ of the subdiagonal of the matrix $D$ agree with the quotients $\gamma_{n}/\gamma_{n+1}$. Since $\displaystyle{\lim_{n\rightarrow \infty } d_{n+1,n}=d_{1}}$, then
\[ d_1 =\lim_{n \to \infty} d_{n+1,n}=\lim_{n \to \infty} \frac{\gamma_{n}}{\gamma_{n+1}} = \lim_{n \to \infty} \frac{1}{\sqrt[n]{\gamma_{n}}}.\]
On the other hand, since $\mu$ is regular, then \cite{totik}
    \[ \lim_{n \to \infty} \frac{1}{\sqrt[n]{\gamma_{n}}} = \capa(\supp(\mu)).\]
Therefore, $d_1 =\capa(\supp(\mu))=\capa (\Gamma )$.
\end{proof}

\begin{coro}\label{diagonals coro}
Let  $D=(d_{ij})_{i,j=1}^{\infty}$ be the Hessenberg matrix associated with a measure $\mu $ with compact support on the complex plane.
 Assume that:
\begin{enumerate}
\item The measure $\mu$ is regular with
$\Gamma =\supp(\mu) $ a Jordan arc or a finite union of Jordan arcs
such that $\CC_\infty \setminus \Gamma $ is a simply connected set of the Riemann sphere $\CC_{\infty} $.
\item The Hessenberg matrix operator $D$ is uniformly asymptotically Toeplitz and the corresponding  limit matrix $T$   has its rows in $\ell ^1$.
\end{enumerate}
Then, the symbol of $T$ is the restriction to the unit circle of the the Riemann mapping function  $\phi:\CC_{\infty} \setminus \overline{\DD }\rightarrow \CC_{\infty}  \setminus \Gamma$ which maps conformally the exterior of the unit disk onto the exterior of the support of the measure $\mu $.
\end{coro}

\begin{proof}
It is well known that every function  $\sum_{k=0}^{\infty}  \left|\dfrac{d_{-k} }{z^{k}}\right| $ with $\sum_{k=0}^{\infty}  \left|d_{-k} \right| \leq \infty$ is continuous on unit circle $|z|=1$.
\end{proof}

As an illustration of Theorem 1 we consider the following example.

\begin{examp}[\bf Arc of circle] \label{arco}
We consider  $\Gamma $ an arc of the unit circle $\TT $. In this case \cite{golinskii} (see also \cite{simon1,simon2}), there exists a regular measure for which the diagonals of the Hessenberg matrix stabilize from the second element on. The monic orthogonal polynomials associated to this measure satisfy $\Psi _0 (0)=1$ and $\Psi _n (0)=\dfrac{1}{a} $ ($a>1$), if $n\geq 1$, and the corresponding Hessenberg matrix is the following unitary matrix $D$
\[
\left(
\begin{array}{cccccc}
-\dfrac{1}{a} & -\dfrac{(a^{2}-1)^{1/2}}{a^{2}} & -\dfrac{(a^{2}-1)^{2/2}}{a^{3}} & -\dfrac{(a^{2}-1)^{3/2}}{a^{4}} & -\dfrac{(a^{2}-1)^{4/2}}{a^{5}}& \cdots \\
\dfrac{(a^{2}-1)^{1/2}}{a} & -\dfrac{1}{a^{2}} & -\dfrac{(a^{2}-1)^{1/2}}{a^{3}} & -\dfrac{(a^{2}-1)^{2/2}}{a^{4}} &  -\dfrac{(a^{2}-1)^{3/2}}{a^{5}} &\cdots \\
0 & \dfrac{(a^{2}-1)^{1/2}}{a} & -\dfrac{1}{a^{2}} & -\dfrac{(a^{2}-1)^{1/2}}{a^{3}} & -\dfrac{(a^{2}-1)^{2/2}}{a^{4}}& \cdots \\
0 & 0 & \dfrac{(a^{2}-1)^{1/2}}{a} & -\dfrac{1}{a^{2}} & \dfrac{(a^{2}-1)^{1/2}}{a^{3}} &\cdots \\
\vdots & \vdots & \vdots & \vdots &  \vdots &\ddots
\end{array}
\right).
\]
Note that $S^{*n}DS^n=T$ for all $n\geq 1$ is a  Toeplitz matrix
such  that the first row is a geometric series of ratio $\dfrac{\sqrt{a^2-1}}{a}<1$ (because $a>1$). Hence the rows of $T$ are in $\ell ^1 $.

On the other hand,
$$ D-T= \left ( \begin{array}{cccc}
-\dfrac{a-1}{a^2} & -\dfrac{(a^2-1)^{1/2}(a-1)}{a^3} &-\dfrac{(a^2-1)(a-1)}{a^4}  & \ldots \\
0 & 0 & 0 & \ldots \\
0 & 0 & 0 & \ldots \\
\vdots & \vdots & \vdots & \ddots
\end{array}
\right ),$$
is compact because it is Hilbert-Schmidt:
\[ \sqrt{\sum_{i,j=1}^{\infty} |a_{ij}|^2} = \sqrt{\sum_{n=0}^{\infty}  \left ( \dfrac{(a-1)(a^2-1)^{n/2}}{a^{2+n}} \right )^2}=
\dfrac{a-1}{a}<+\infty.\]

Then we have that $D=T+K$ where $T$ is a Toeplitz operator and $K$ is compact, therefore $D$ is asymptotically uniformly Toeplitz \cite{feintuch} and the hypothesis of  Theorem 1 are satisfied. Then, the  expression of the Riemann mapping function as a Laurent series is
 \begin{align*}
\phi(z)
&=\dfrac{z\left( a-\sqrt{a^{2}-1} \ z\right) }{\sqrt{a^{2}-1}-az}
=\dfrac{\sqrt{a^2-1}}{a} z - \sum_{n=0}^{\infty } \frac{(a^2-1)^{\frac{n}{2}}}{a^{n+2} }z^{-n} .
\end{align*}

\begin{figure}[H]
\begin{center}
\includegraphics[width=5cm,keepaspectratio=true]{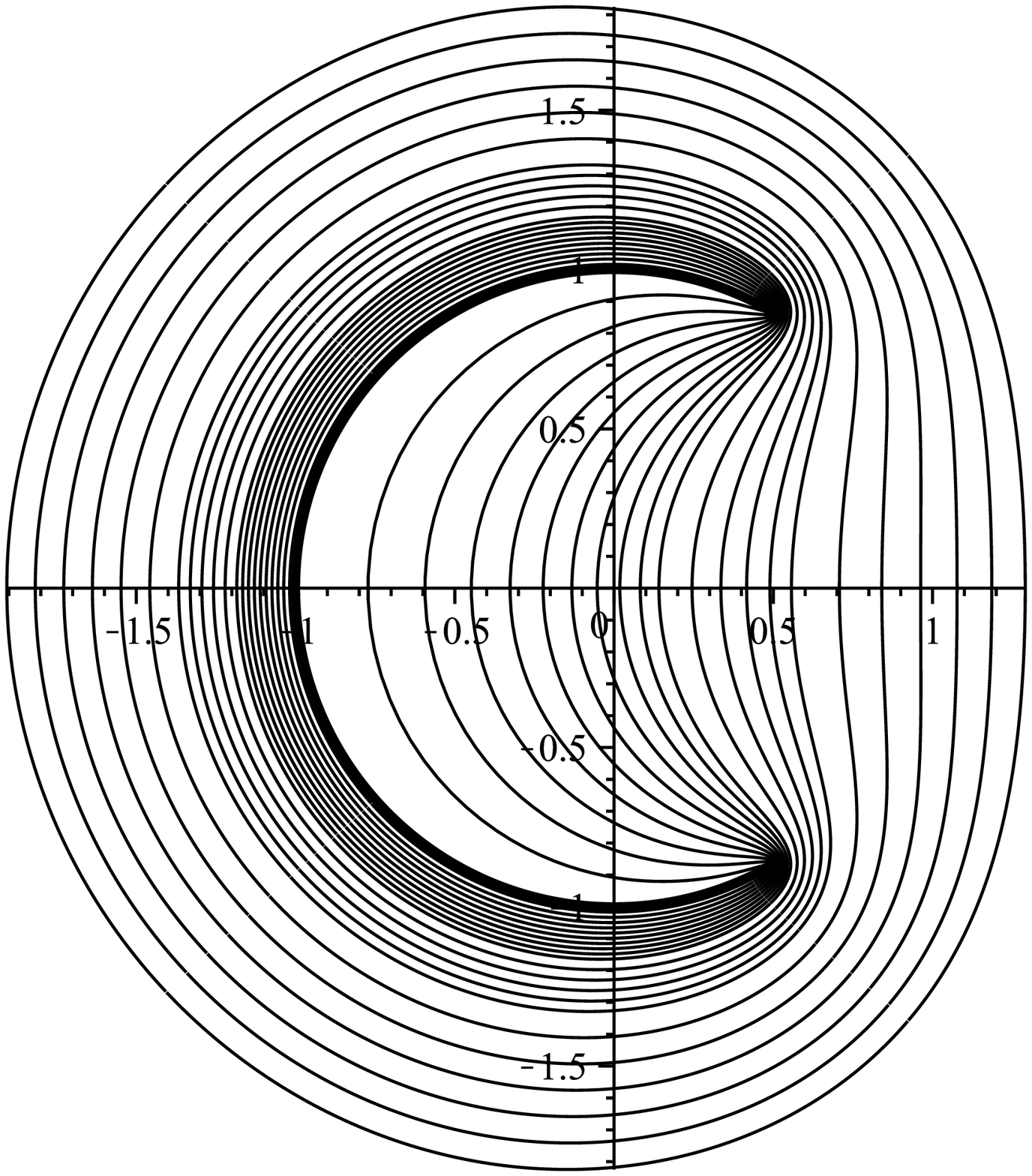}
\includegraphics[width=5cm,keepaspectratio=true]{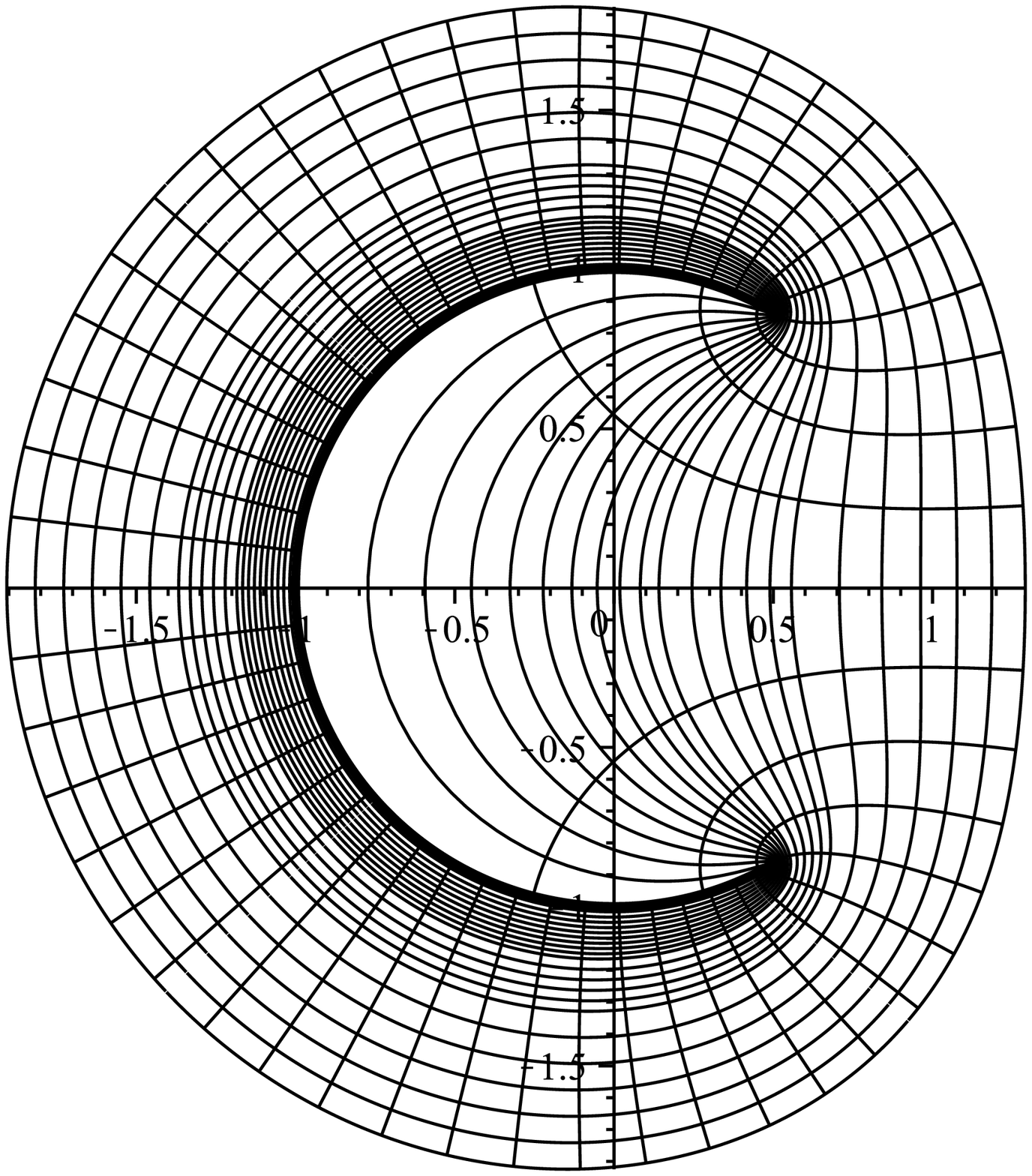}
\caption{The Riemann mapping function for the arc for $a=2$}
\end{center}
\end{figure}
\end{examp}

\section{Approximation of the Riemann mapping function}

When the Hessenberg matrix $D$ can not obtained as a closed form and it is not possible to compute the limits of the elements of its diagonals, we may ask if it is still possible to compute approximations of
the Riemann mapping function.

Specifically, under the hypothesis of theorem \ref{diagonals th}, since the coefficients of the Riemann mapping function are the limits of the elements in each of the diagonals of the Hessenberg matrix, we may ask if  the functions
   \[h_{n}(z)= d_{n+1,n}z + d_{n,n}+ \frac{d_{n-1,n}}{z} + \frac{d_{n-2,n}}{z^2} + \ldots + \frac{d_{1,n}}{z^{n-1}},\]
defined from the $n$-th column $c_n$ of  the Hessenberg matrix, are suitable approximations of the Riemann mapping function $\phi (z)$.

We show in this section that this is indeed the case.

In what follows we will denote by $\Theta _n $ the norm in $\ell ^2$ the of $n$-th column of the matrix $D-T$ as a vector of $\ell ^2$, i.e.,
\[ \Theta_{n} = \sqrt{\sum_{k=-1}^{n-1}  |d_{-k}-d_{n-k,n}|^{2}}.\]

\begin{lem}
Suppose that $D$ is an asymptotically uniformly Toeplitz  operator on $\ell^2$. Then
\[ \lim _{n\to \infty } \Theta _n =  0. \]
\end{lem}

\begin{proof}
Note that  $D-T$ is a compact operator
\cite{feintuch}.
 Since  $\{ e_{n} \} $ is a  sequence   weakly convergent to 0, then $\{ (D-T) e_{n-1} \} $ converges strongly to 0.

Therefore,
$ \Theta _n  =  \| (D-T)e_{n-1}\|_2 \to 0$.
\end{proof}

\begin{thm} \label{app unif}
Under the hypothesis of Theorem \ref{diagonals th},
the sequence of functions
$$h_n(z)=d_{n+1,n}z + d_{n,n}+ \frac{d_{n-1,n}}{z} + \frac{d_{n-2,n}}{z^2} + \ldots + \frac{d_{1,n}}{z^{n-1}}$$
converges uniformly to the Riemann mapping function  $\phi(z)$  on any compact set $K \subset \CC \setminus \overline{\DD }$.
\end{thm}

\begin{proof}
Consider a compact subset $K\subset \CC \setminus \overline{\DD }$ and consider $\varepsilon >0$. Since $K$ is compact, there exist $r,R\in \RR $ such that $1<r\leq |z|\leq R$ for every $z\in K$.

For every $z\in K$, we have
\begin{align}
|h_{n}(z)-\phi(z)|&= \notag \\
=&\left| (d_{1}-d_{n+1,n})z + (d_{0}-d_{n,n})+ \frac{d_{-1}-d_{n-1,n}}{z}+\cdots \right. \notag \\
&\left. \dots +\frac{d_{1-n}-d_{1,n}}{z^{n-1}} + \sum_{k=n}^{\infty} d_{-k} \frac{1}{z^{-k}} \right| \notag \\
\leq & \sum_{k=-1}^{n-1} \left|\dfrac{d_{-k}-d_{n-k,n}}{z^{k} }\right| + \sum_{k=n}^{\infty} \left|\dfrac{d_{-k} }{z^{k}}\right|
 \label{desigual} .
\end{align}

Applying  Cauchy-Schwarz inequality to the second summand of inequality \ref{desigual}, we have
$$\sum_{k=n}^{\infty}  \left|\dfrac{d_{-k} }{z^{k}}\right| \leq  \sqrt{\sum_{k=n}^{\infty} |d_{-k}|^{2}} \;
\sqrt{\sum_{k=n}^{\infty}|z|^{-2k}}
\leq \sqrt{ \frac{1}{r^2-1} }\sqrt{\sum_{k=n}^{\infty} |d_{-k}|^{2}}
.$$ \label{theta}
The last factor is the tail of a vector in $\ell ^2$. Therefore, given $\varepsilon >0$, there exists $N_0 \in \NN $ such that, for every $n<N_0$,
$\sqrt{\sum_{k=n}^{\infty} |d_{-k}|^2}<\sqrt{r^2-1} \dfrac{\varepsilon }{2}$.

In a similar way, for the first summand of inequality (\ref{desigual}), we have
$$ \sum_{k=-1}^{n-1} \left|\dfrac{d_{-k}-d_{n-k,n}}{z^{k} }\right| \leq  \sqrt{\sum_{k=-1}^{n-1}  |d_{-k}-d_{n-k,n}|^{2}} \;
\sqrt{\sum_{k=-1}^{n-1}|z|^{-2k}} =\Theta _n \; \sqrt{\sum_{k=-1}^{n-1}|z|^{-2k}} .$$ \label{theta}
Then, for every $z\in K$,
\[ \sum_{k=-1}^{n-1}|z|^{-2k} \leq R^2+\sum_{k=0}^{\infty}r^{-2k} = R^2+\frac{1}{r^2-1}=C. \]
Since $\Theta _n$ converges to zero, there exists $N_1 \in \NN $ such that, for every $n>N_1$, $\Theta _n <\dfrac{\varepsilon }{2\sqrt{C}}$.

Taking $N=\max \{N_0,N_1\}$ we have that
\[ |h_{n}(z) - \phi(z)|< \epsilon\]
for every $z \in K$, for every $n>N$.
\end{proof}

\begin{remark}

Under the hypothesis of Theorem \ref{diagonals th} we consider the sequence $\{c_n\}_{n=1}^\infty $ of column vectors of the matrix $D-T$.  Since every $c_n$ has at most $n$ non null elements, we can calculate its norm in $\ell ^1$ and in $\ell^2 $. We denote these norms by $\theta _n$ and $\Theta _n$:
\begin{align*}
\theta n&=\| c_n\|_1=\sum _{k=-1}^n |d_{n-k,n}-d_{-k}|,\\
\Theta n&=\| c_n\|_2=\sqrt{\sum _{k=-1}^n |d_{n-k,n}-d_{-k}|^2}.
\end{align*}

Theorem \ref{app unif} assures the   uniform convergence of $h_n(z)$ to the Riemann mapping function  $\phi(z)$  on any compact set $K \subset \CC \setminus \overline{\DD }$.

Now we analyse this convergence when  $z=re^{i\theta}$ for $r\geq 1$.

\begin{itemize}
\item[i)] For $r>1$, applying Cauchy-Schwarz inequality we have

\begin{align*}
|h_{n}(re^{i\theta})-\phi(re^{i\theta})|& \leq  \sum_{k=-1}^{n-1} \left|\dfrac{d_{-k}-d_{n-k,n}}{z^{k} }\right| + \sum_{k=n}^{\infty} \left|\dfrac{d_{-k} }{z^{k}}\right| \\
& \leq  \Theta _n \; \sqrt{r^2+\frac{1}{r^2-1}} + ||\widetilde{\phi }||_2  \frac{r}{\sqrt{r^2-1}}\;  \frac{1}{r^n}
 \label{desigual} .
\end{align*}

Therefore, on each circle of radius $r$, the order of convergence of $h_n$ to $\phi $ will be at least that of $\Theta _n$ to zero because the other summand is of the order of $\frac{1}{r^n}$, $r>1$.

\item [ii)] For $r=1$, we need some additional hypothesis. Under the hypothesis of Corollary \ref{diagonals coro},
the coefficients of $\phi $ are in $\ell ^1$. Now we have

 \begin{equation*} \label{hnphi}
 |h_{n}(e^{i\theta})-\phi(e^{i\theta})| \leq \sum_{k=-1}^{n-1} \left|d_{-k}-d_{n-k,n}\right| + \sum_{k=n}^{\infty} \left|{d_{-k} }\right| \leq \theta _n+\sum_{k=n}^{\infty} \left|{d_{-k} }\right|
 \end{equation*}

The last summand is the tail of a vector in $\ell ^1$ which goes to zero. Then, the convergence  $\theta _n\too 0$   assures the convergence of $h_n \too \phi $ on the unit circle.
\end{itemize}
\end{remark}

\begin{thm}\label{app Kr}
 Under the hypothesis of Theorem 1, for every $\varepsilon >0$ there is a $\delta >0$ such that for every $r\in (1, \delta +1)$ there is a natural number $N$, such that for all $n\geq N$ we have
$$ |h_n(re^{i\theta })-\phi (e^{i\theta }) | \leq \varepsilon ,$$
for every $\theta \in [0,2\pi ]$.
\end{thm}

\begin{proof}
Consider $\varepsilon >0$. Consider the compact set  $K=\{ z : 1\leq |z|\leq 2\}$. Since the function $\phi (z)$ is uniformly continuous in $K$, there exists  $\delta >0$, such that for every $r\in [1,1+\delta)$, we have
$$ |\phi(re^{i\theta })-\phi (e^{i\theta })| < \frac{\varepsilon}{2} ,$$
for every $\theta \in [0,2\pi ]$.

We consider now, for every $r\in (1,1+\delta )$, the compact set $K_r=\{ z\in \CC :|z|=r\} $. By the uniform convergence of $h_n $ on compact sets of $\CC \setminus \overline{\DD }$, established in Proposition 1,
there exists $n \in \NN $ such that, for every $n\geq N$, we have
 $$ |h_n(re^{i\theta })-\phi (re^{i\theta })| < \frac{\varepsilon}{2} $$
for all $\theta \in [0,2\pi ]$.
Then,
\begin{equation*}
 |h_n(re^{i\theta })-\phi (e^{i\theta })|\leq |h_n(re^{i\theta })-\phi (re^{i\theta })|+ |\phi (re^{i\theta })-\phi (e^{i\theta })|
  \leq \varepsilon
\end{equation*}
for every $\theta \in [0,2\pi ]$
\end{proof}

We will use this result, in the last section, to approximate the support of the measure by equipotential curves of the function $h_n(re^{i\theta})$, for  suitable $n$ and $r$.

\begin{examp}[\bf Arc of circle revisited]
Consider again the arc of circle $\Gamma $ of Example 1  where we proved that it satisfies the hypothesis of Theorem \ref{diagonals th}  and hence
\begin{align*}
\phi(z)
&=\dfrac{z\left( a-\sqrt{a^{2}-1} \ z\right) }{\sqrt{a^{2}-1}-az}
=\dfrac{\sqrt{a^2-1}}{a} z - \sum_{n=0}^{\infty } \frac{(a^2-1)^{\frac{n}{2}}}{a^{n+2} }z^{-n} .
\end{align*}
Note that, since $a>1$, $\sqrt{1-\dfrac{1}{a^2}}<1$ and the sequence of coefficients of $\phi $  is of order $\mathrm{O}  \left( \left(\sqrt{1-\dfrac{1}{a^2}}\right)^n \right) $ so it  is both in $\ell ^1 $ and $\ell ^2$.

We also showed that, in this case, the $n$-th column of $D-T$ reduces to a single element and then
\[ \Theta_{n}= \theta_{n}= \frac{a-1}{a^2} \left (\sqrt{1-\frac{1}{a^2}} \right )^{n} = \mathrm{O} \left( \left(\sqrt{1-\dfrac{1}{a^2}}\right)^n \right) .\]

Then, for every $z\in \CC \setminus \DD $,
\begin{align*}
|h_{n}(z)- \phi (z)| &\leq
\sum_{k=-1}^{n-1} \left|\dfrac{d_{-k}-d_{n-k,n}}{z^{k} }\right| + \sum_{k=n}^{\infty} \left|\dfrac{d_{-k} }{z^{k}}\right| \\
& \leq
\theta _n + \sum_{k=n}^{\infty} |d_{-k} |
= \mathrm{O} \left(\left(\sqrt{1-\dfrac{1}{a^2}}\right)^n  \right) .
\end{align*}

In particular, if $a=2$,
$$|h_{n}(z)- \phi (z)|
=\dfrac{5+2\sqrt{3}}{4} \left (\frac{\sqrt{3}}{2} \right )^{n} , \text{ for all }
|z|\geq 1.$$

Note that the difference between $\phi (z) $ and $h_n(z) $ decreases as $|z|$ increases and the worse approximation appear when $z$ is on the unit circle.

This inequality allows us to calculate the  necessary value of  $n$ to obtain a desired approximation of $\phi $. Some values can be seen in the following table:
\begin{center}
\begin{tabular}{|c|c|}
\hline
Bound of $||h_n |_{\TT }-\phi |_{\TT }||_{\infty } $ & $n$ \\ \hline
$0.2 $ & 17\\ \hline
$0.1 $ & 22\\ \hline
$0.01 $ & 38\\ \hline
$0.001 $ & 54\\ \hline
$0.0001 $ & 70\\ \hline
\end{tabular}
\end{center}

In the following figure we show the graphical result of approximating $\phi (\TT )$ using $h_n (\TT )$.
\begin{figure}[H]
\begin{center}
\includegraphics[width=4cm,keepaspectratio=true]{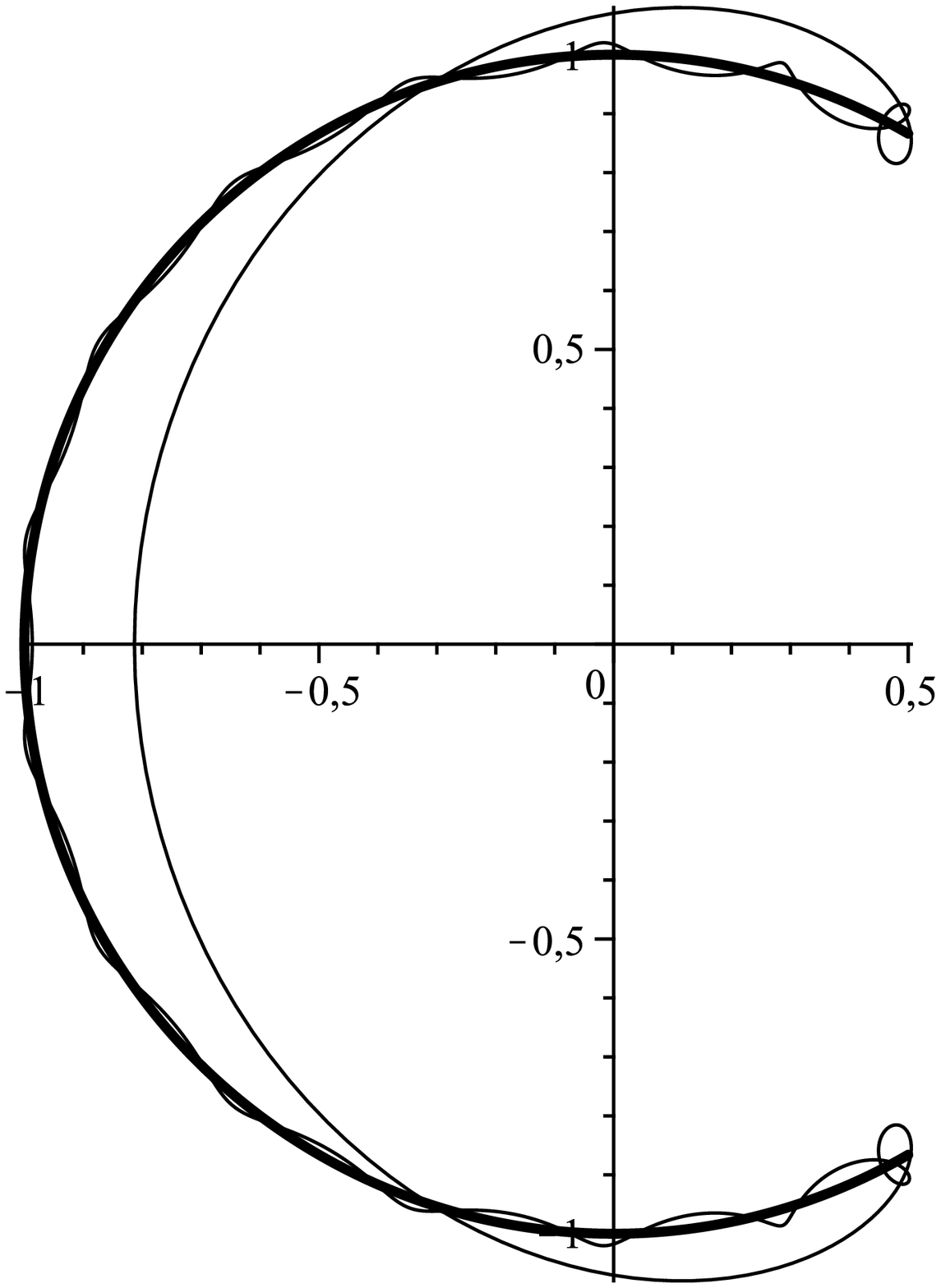}
\includegraphics[width=4cm,keepaspectratio=true]{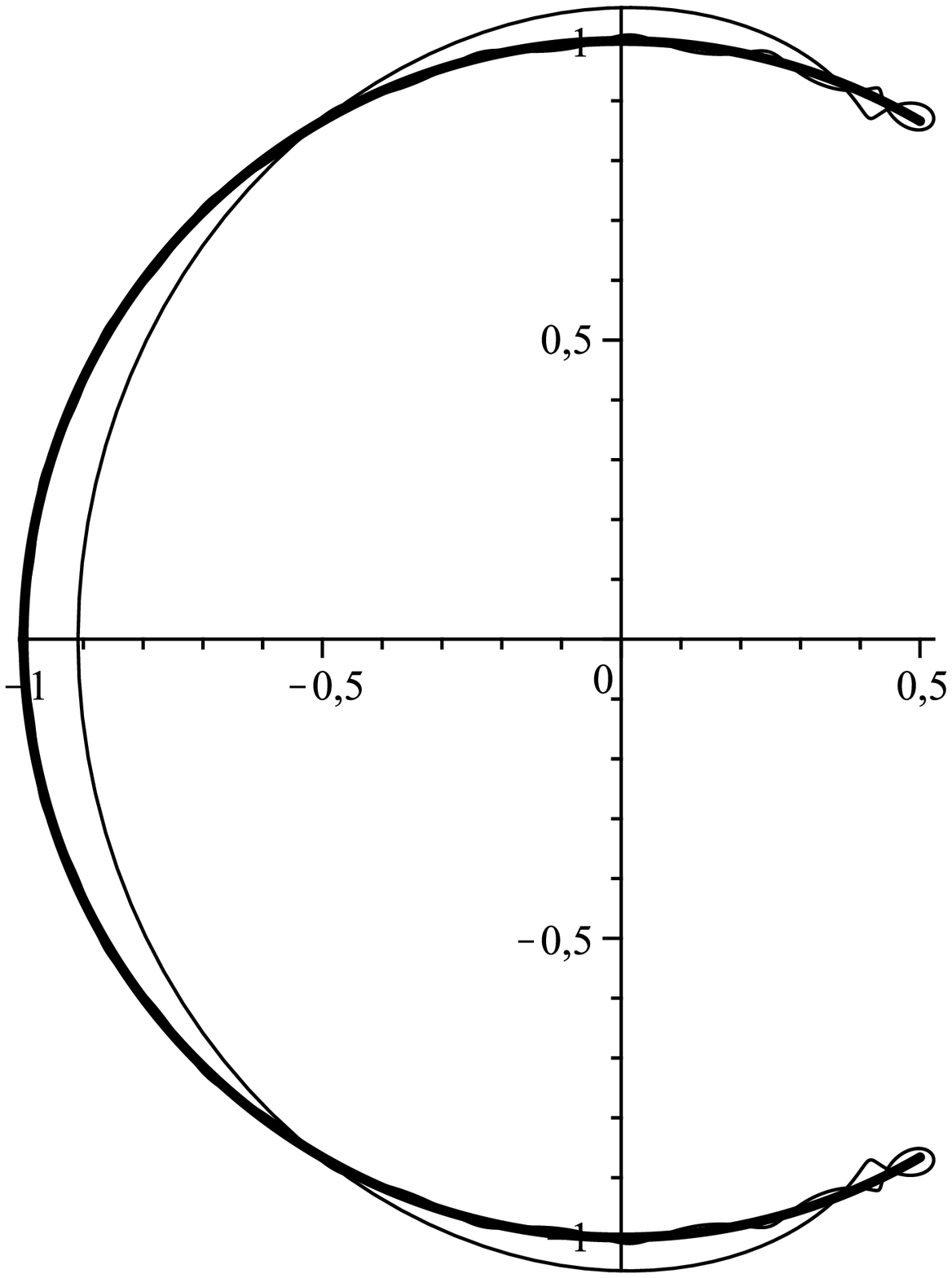}
\includegraphics[width=4cm,keepaspectratio=true]{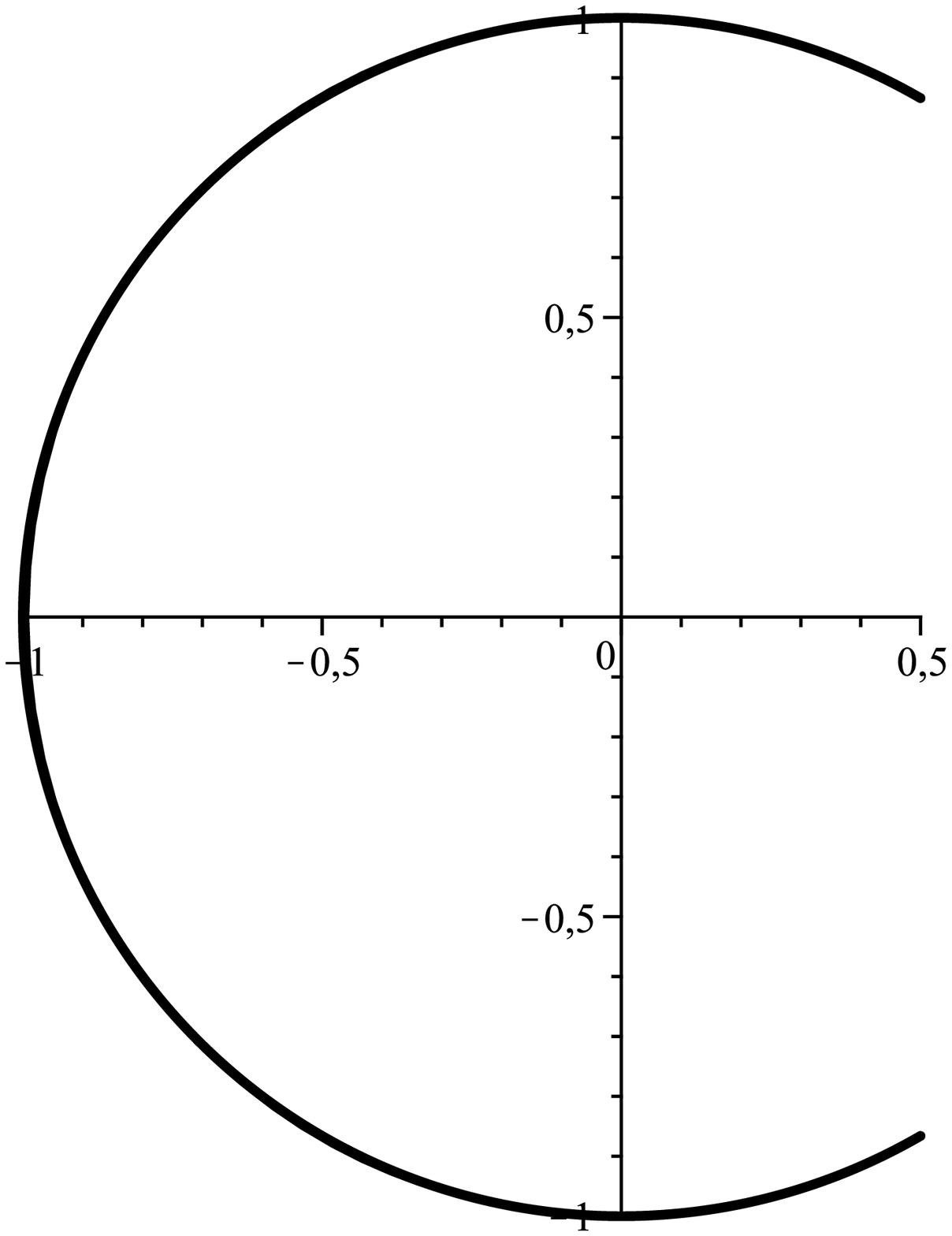}
\caption{$h_n (\TT )$ for $n=16$, $n=21$ and $n=37$, respectively}
\end{center}
\end{figure}

In the following figure we show the graphical result of approximating the equipotential lines of $\phi (re^{i\theta })$  using $h_n (re^{i\theta } )$:
\begin{figure}[H]
\begin{center}
\includegraphics[width=4cm,keepaspectratio=true]{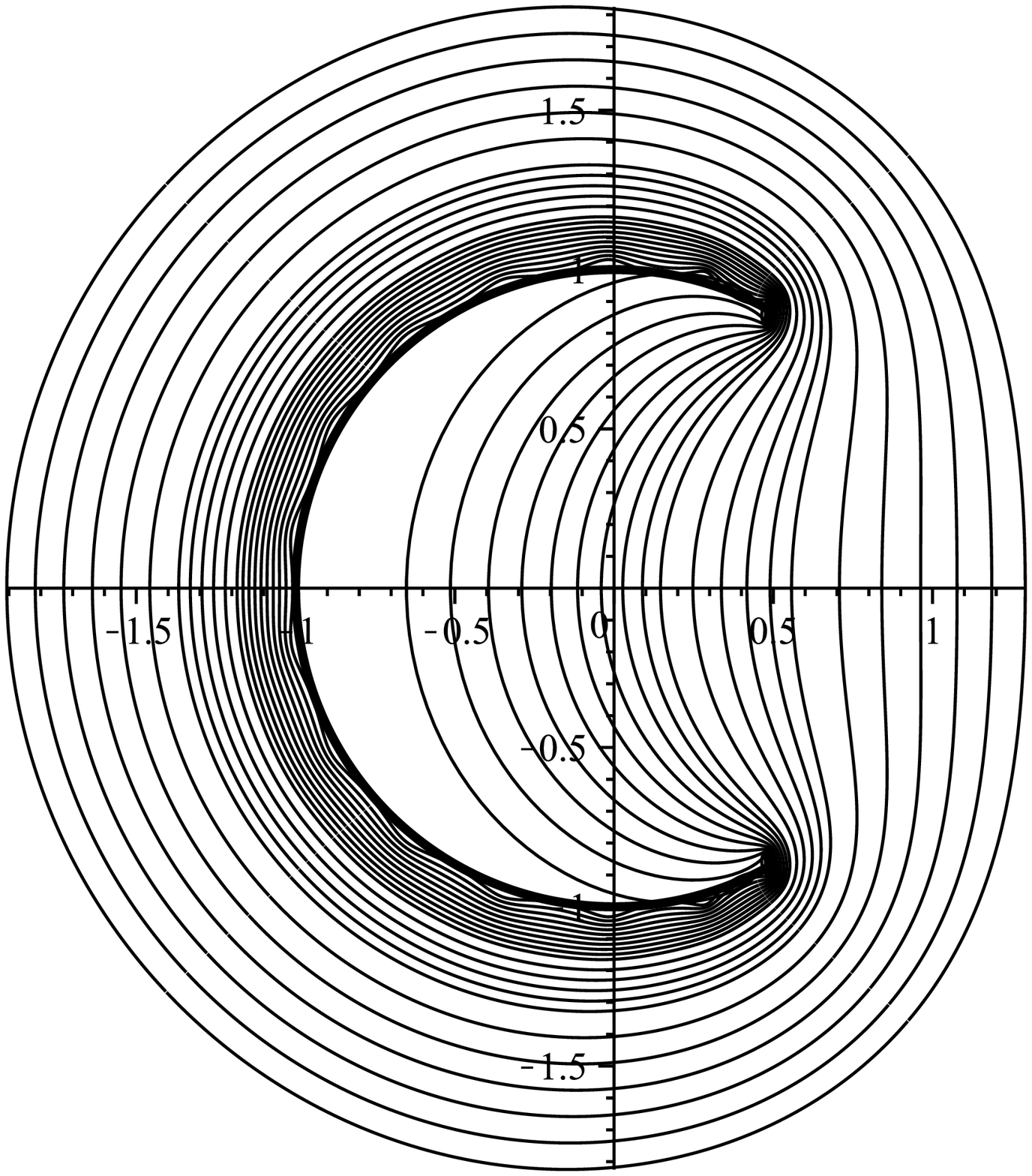}
\includegraphics[width=4cm,keepaspectratio=true]{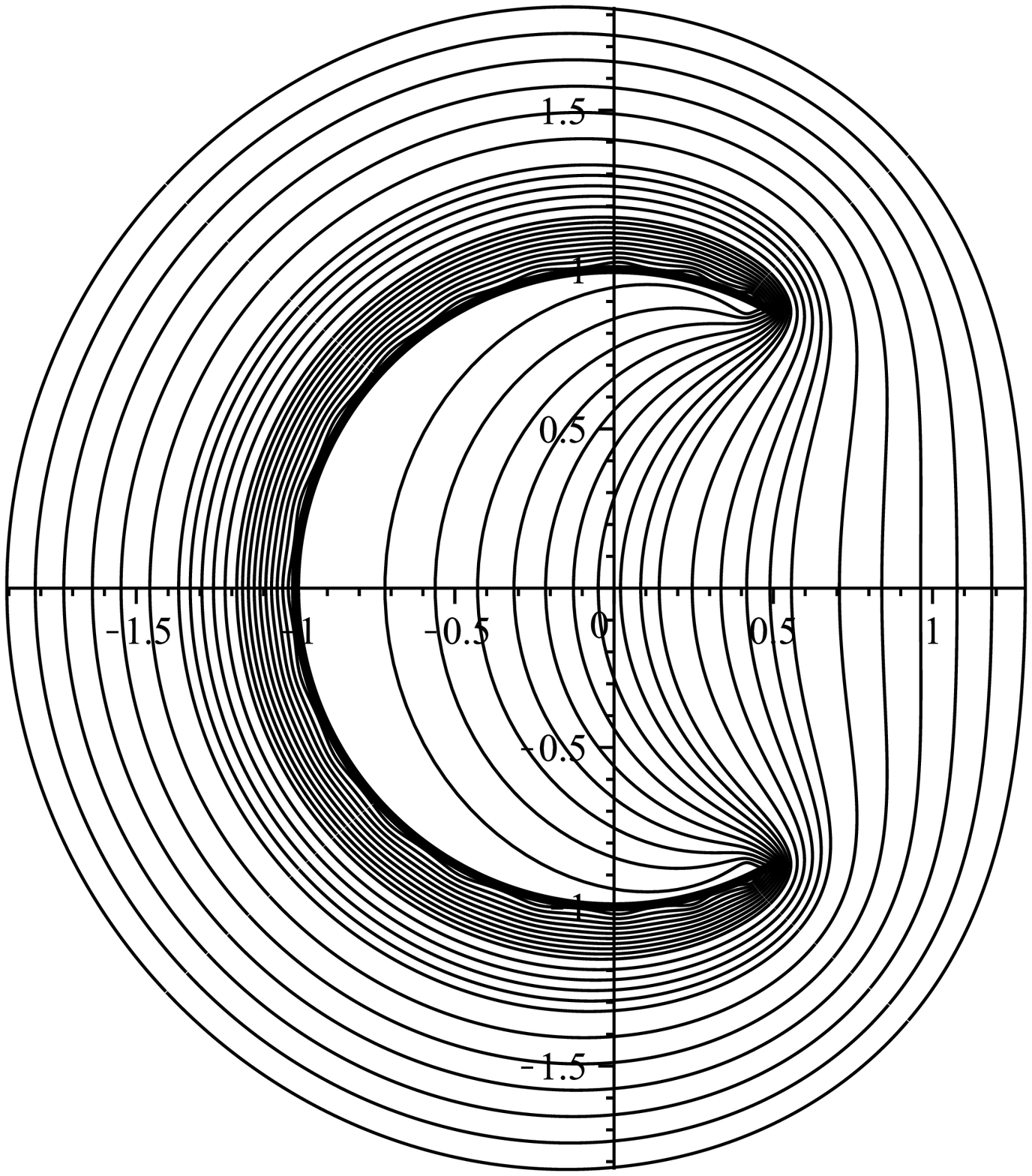}
\includegraphics[width=4cm,keepaspectratio=true]{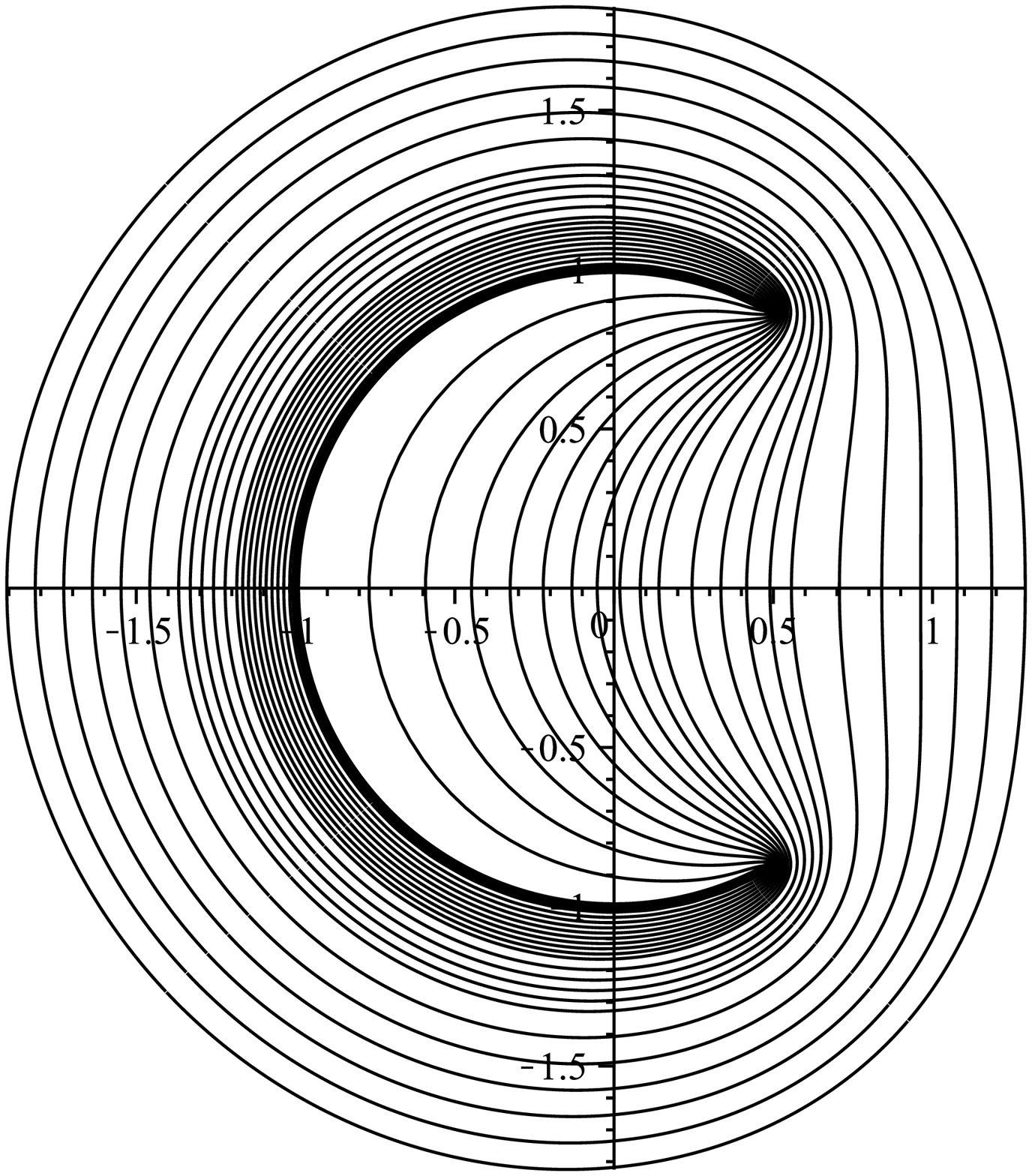}
\caption{$h_n (S_r )$ for some values of $r\in (1,1.5]$, for $n=16$, $n=21$ and $n=37$, respectively}
\end{center}
\end{figure}
\end{examp}

\section{ Numerical examples}

In this section, we present some numerical experiments using the results from the previous sections on the approximation of the Riemann mapping function.
\begin{examp}
Let $\Gamma $ be a cross-like set formed by the intervals $[-a,a]$ and $[-bi,bi]$, with $a,b \in (0, \infty)$, and let $\mu $ be the uniform measure on $\Gamma $. The Riemann mapping function
deduced from
(\cite[pg 118]{lavrentiev}) is
$$\phi(z)= \dfrac{\sqrt{a^2(z^2+1)^2+b^2(z^2-1)^2}}{2z}.$$
In the particular case of $a=b$,
$$\phi(z)=\dfrac{a\sqrt{2}}{2z}\sqrt{z^4+1}.$$

The Laurent series expansion of $\phi (z)$ in a neighbourhood of infinity is
\[\phi(z)=
\dfrac{\sqrt{a^{2} + b^{2}}}{2} \; z  +\dfrac {- 2\,b^{2} + 2\,a^{2}}{4\,\sqrt{a^{2} + b^{2}}} \; \dfrac{1}{z} + \dfrac{\sqrt{a^{2} + b^{2}}\left(\dfrac {1}{2} - \dfrac{( - 2\,b^{2} + 2\,a^{2})^{2}}{8\,(a^{2} + b^{2})^{2}} \right) }{2\,z^{3}} +\mathrm{O}\left(\dfrac {1}{z^{5}} \right).\]
Note that the first coefficient $\dfrac{\sqrt{a^{2} + b^{2}}}{2} $ agrees with the capacity of the support.

If $a=b=1$, the series expansion is
 $$\phi(z)=\frac{1}{\sqrt{2}}\sum _{n=0}^\infty \binom {\frac{1}{2}}{n} z^{1-4n}.$$
It is easy to show that the sequence of the coefficients of $\phi $ is in $\ell ^1$.

In the following image we represent the Riemann mapping function for $\Gamma $.
\begin{figure}[H]
\begin{center}
\includegraphics[width=5cm,keepaspectratio=true]{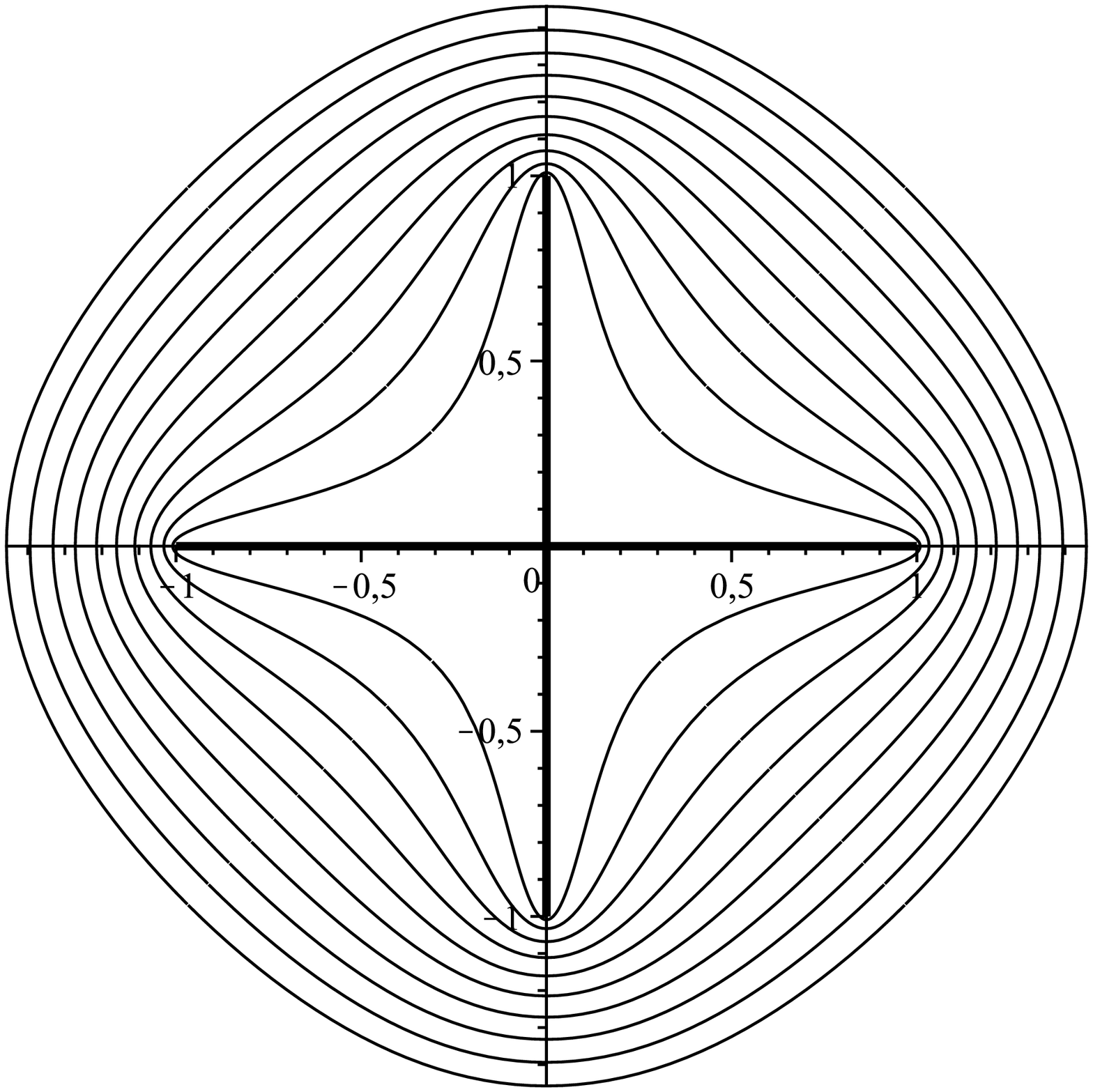}
\includegraphics[width=5cm,keepaspectratio=true]{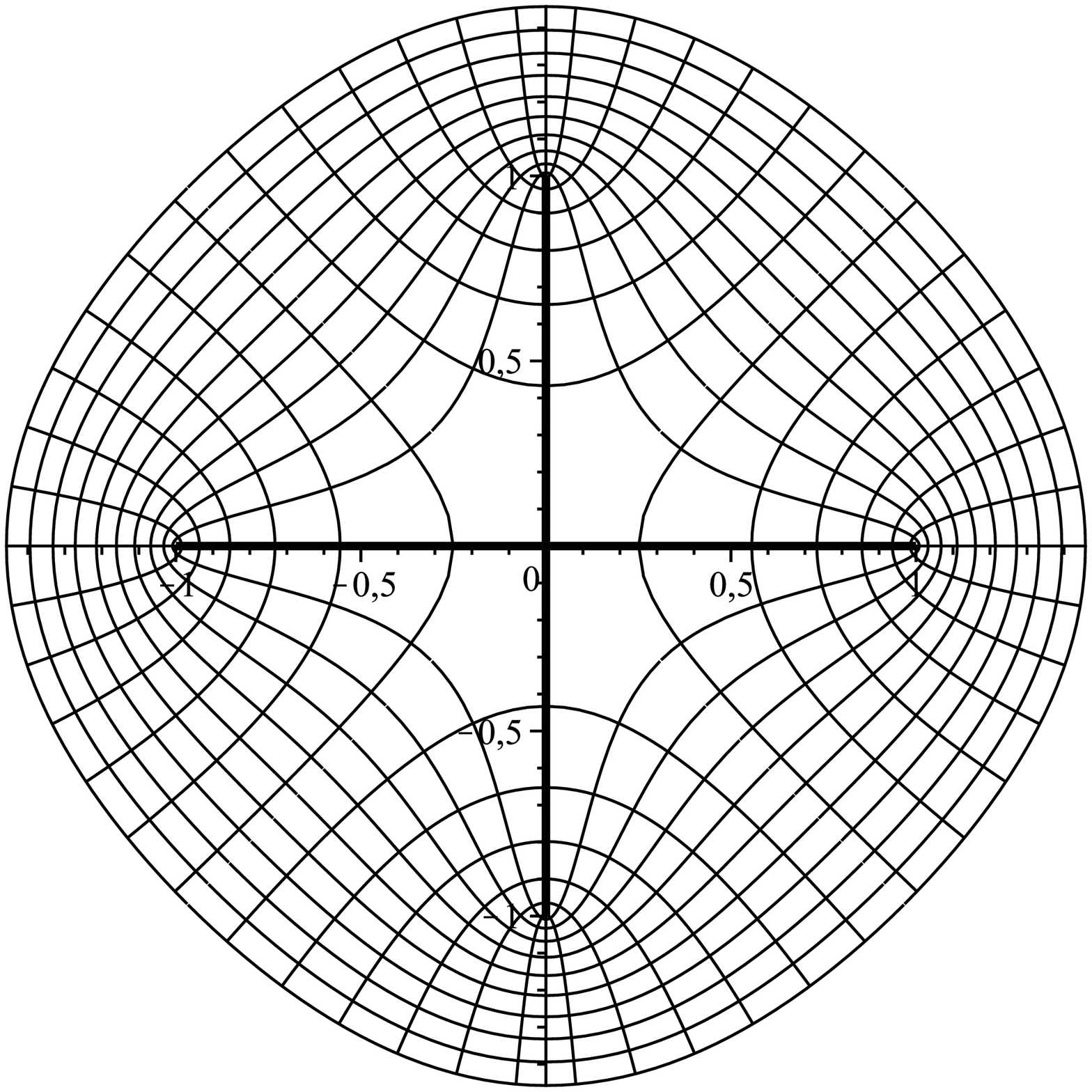}
\caption{The Riemann mapping function $\phi (z )$ for a cross-like set}
\end{center}
\end{figure}

The 9-th section of the Hessenberg matrix of $\mu $, obtained from the moment matrix, is
$$\left( \tiny{\begin {array}{cccccccccc}
0 \hspace{-0.2cm} & 0 \hspace{-0.2cm} & 0 \hspace{-0.2cm} & \dfrac{\sqrt{7}}{5} \hspace{-0.2cm} & 0 \hspace{-0.2cm} & 0 \hspace{-0.2cm} & 0 \hspace{-0.4cm} &-\dfrac{2\sqrt{15}}{45} \hspace{-0.4cm} & 0 \\
\dfrac{\sqrt{3}}{3} \hspace{-0.2cm} & 0 \hspace{-0.2cm} & 0 \hspace{-0.2cm} & 0 \hspace{-0.2cm} & \dfrac{2\sqrt{3}}{5} \hspace{-0.2cm} & 0 \hspace{-0.2cm} & 0 \hspace{-0.4cm} & 0 \hspace{-0.4cm} & -\dfrac{4\sqrt{3}\sqrt{17}}{231} \\
0 \hspace{-0.2cm} & \dfrac{\sqrt{5}\sqrt{3}}{5} \hspace{-0.2cm} & 0 \hspace{-0.2cm} & 0 \hspace{-0.2cm} & 0 \hspace{-0.2cm} & \dfrac{2\sqrt{5}\sqrt{11}}{45} \hspace{-0.2cm} & 0 \hspace{-0.4cm} & 0 \hspace{-0.4cm} & 0 \\
0 \hspace{-0.2cm} & 0 \hspace{-0.2cm} & \dfrac{\sqrt{7}\sqrt{5}}{7} \hspace{-0.2cm} & 0 \hspace{-0.2cm} &
0 \hspace{-0.2cm} & 0 \hspace{-0.2cm} & {\dfrac{2\sqrt{7}\sqrt{13}}{77}} \hspace{-0.4cm} & 0 \hspace{-0.4cm} & 0 \\
0 \hspace{-0.2cm} & 0 \hspace{-0.2cm} & 0 \hspace{-0.2cm} & \dfrac{4\sqrt{7}}{15} \hspace{-0.2cm} & 0 \hspace{-0.2cm} & 0 \hspace{-0.2cm} & 0 \hspace{-0.4cm} & \dfrac{19\sqrt{15}}{195} \hspace{-0.4cm} & 0 \\
0 \hspace{-0.2cm} & 0 \hspace{-0.2cm} & 0 \hspace{-0.2cm} & 0 \hspace{-0.2cm} & \dfrac{15\sqrt{11}}{77} \hspace{-0.2cm} & 0 \hspace{-0.2cm} & 0 \hspace{-0.4cm} & 0 \hspace{-0.4cm} & \dfrac{12\sqrt{11}\sqrt{17}}{385} \\
0 \hspace{-0.2cm} & 0 \hspace{-0.2cm} & 0 \hspace{-0.2cm} & 0 \hspace{-0.2cm} & 0 \hspace{-0.2cm} & \dfrac{7\sqrt{13}\sqrt{11}}{117} \hspace{-0.2cm} & 0 \hspace{-0.4cm} & 0 \hspace{-0.4cm} & 0 \\
0 \hspace{-0.2cm} & 0 \hspace{-0.2cm} & 0 \hspace{-0.2cm} & 0 \hspace{-0.2cm} & 0 \hspace{-0.2cm} & 0 \hspace{-0.2cm} & \dfrac{3\sqrt{15}\sqrt{13}}{55} \hspace{-0.4cm} & 0 \hspace{-0.4cm} & 0 \\
0 \hspace{-0.2cm} & 0 \hspace{-0.2cm} & 0 \hspace{-0.2cm} & 0 \hspace{-0.2cm} & 0 \hspace{-0.2cm} & 0 \hspace{-0.4cm} & 0\hspace{-0.4cm} & \dfrac{88\sqrt{17}\sqrt{15}}{1989} \hspace{-0.2cm} & 0
\end {array}} \right)
$$
In this case, a closed form for the Hessenberg matrix it is not known and it is not easy  to compute the limits of the diagonals of $D$. However, it is still possible to compute approximations of the Riemann mapping function. Specifically, if the coefficients of the Riemann mapping function are the limits of the elements in the diagonals of the Hessenberg matrix, we may consider, as approximations of the Riemann mapping function $\phi (z)$, the functions $h_n(z)$.

As opposed to the case of the arc in Example 1, we have not here an explicit formula for $\Theta _n $ and $\theta _n $, used there to estimate the degree of approximation obtained using $h_n (z)$ instead of $\phi (z)$ (see Remark 1).
In the following table we give a list of numerical values of $\Theta _n $ and $\theta _n $:
\begin{center}
\begin{tabular}{|c|c|c|c|c|c|}
\hline
$n$ & $\Theta _n $    & $\theta _n $   & $n$ & $\Theta _n $    & $\theta _n $   \\ \hline
 4  & 0.1756039179    & 0.1771699698   &  52 & 0.1435839520e-1 & 0.374355145e-1 \\ \hline
 8  & 0.8706648269e-1 & 0.1081557877   &  56 & 0.1335920853e-1 & 0.359786966e-1 \\ \hline
 12 & 0.5894618764e-1 & 0.846332410e-1 &  60 & 0.1249073290e-1 & 0.346766415e-1 \\ \hline
 16 & 0.4475241502e-1 & 0.716638451e-1 &  64 & 0.1172882241e-1 & 0.335039558e-1 \\ \hline
 20 & 0.3613474685e-1 & 0.631649554e-1 &  68 & 0.1105494437e-1 & 0.324406982e-1 \\ \hline
 24 & 0.3032967468e-1 & 0.570537158e-1 &  72 & 0.1045463567e-1 & 0.314709643e-1 \\ \hline
 28 & 0.2614682972e-1 & 0.523932383e-1 &  76 & 0.9916442395e-2 & 0.305818978e-1 \\ \hline
 32 & 0.2298656524e-1 & 0.486911124e-1 &  80 & 0.9431174829e-2 & 0.297629768e-1 \\ \hline
 36 & 0.2051319544e-1 & 0.456605530e-1 &  84 & 0.8991373805e-2 & 0.290054994e-1 \\ \hline
 40 & 0.1852386296e-1 & 0.431218007e-1 &  88 & 0.8590921072e-2 & 0.283021988e-1 \\ \hline
 44 & 0.1688863030e-1 & 0.409557583e-1 &  92 & 0.8224750644e-2 & 0.276469516e-1 \\ \hline
 48 & 0.1552035810e-1 & 0.390800153e-1 &  96 & 0.7888631635e-2 & 0.270345579e-1 \\ \hline \end{tabular}
\end{center}
We consider values of $n$ that are multiples of 4 because for these values of $n$ the approximations are worse since the matrix $D-T$ has three of every four diagonals nulls.

Some result of approximating $\supp (\mu )$ and the Riemann mapping function using this method, are shown in the followings figures.

\begin{figure}[H]
\begin{center}
\includegraphics[width=4.0cm,keepaspectratio=true]{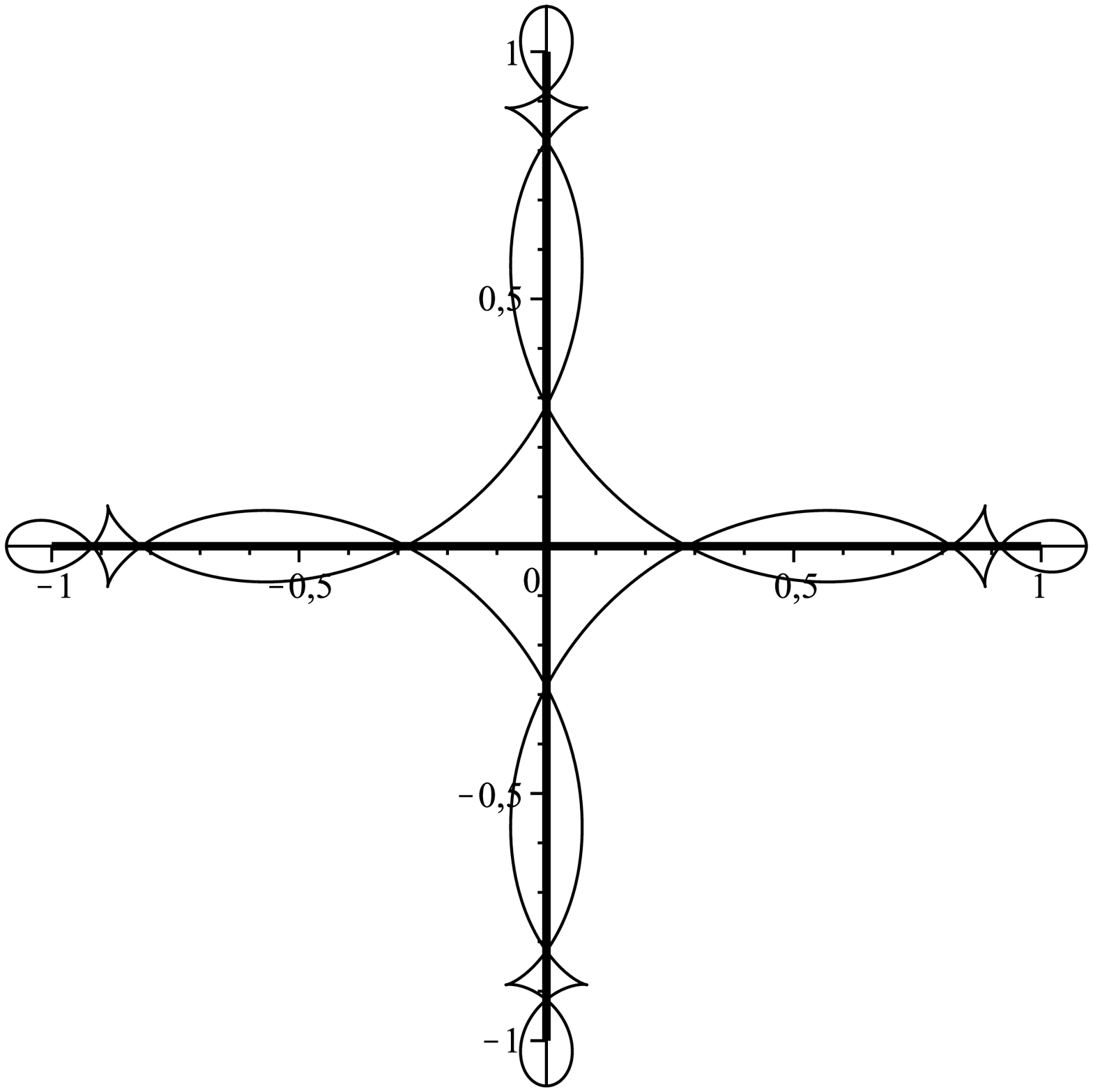}
\includegraphics[width=4.0cm,keepaspectratio=true]{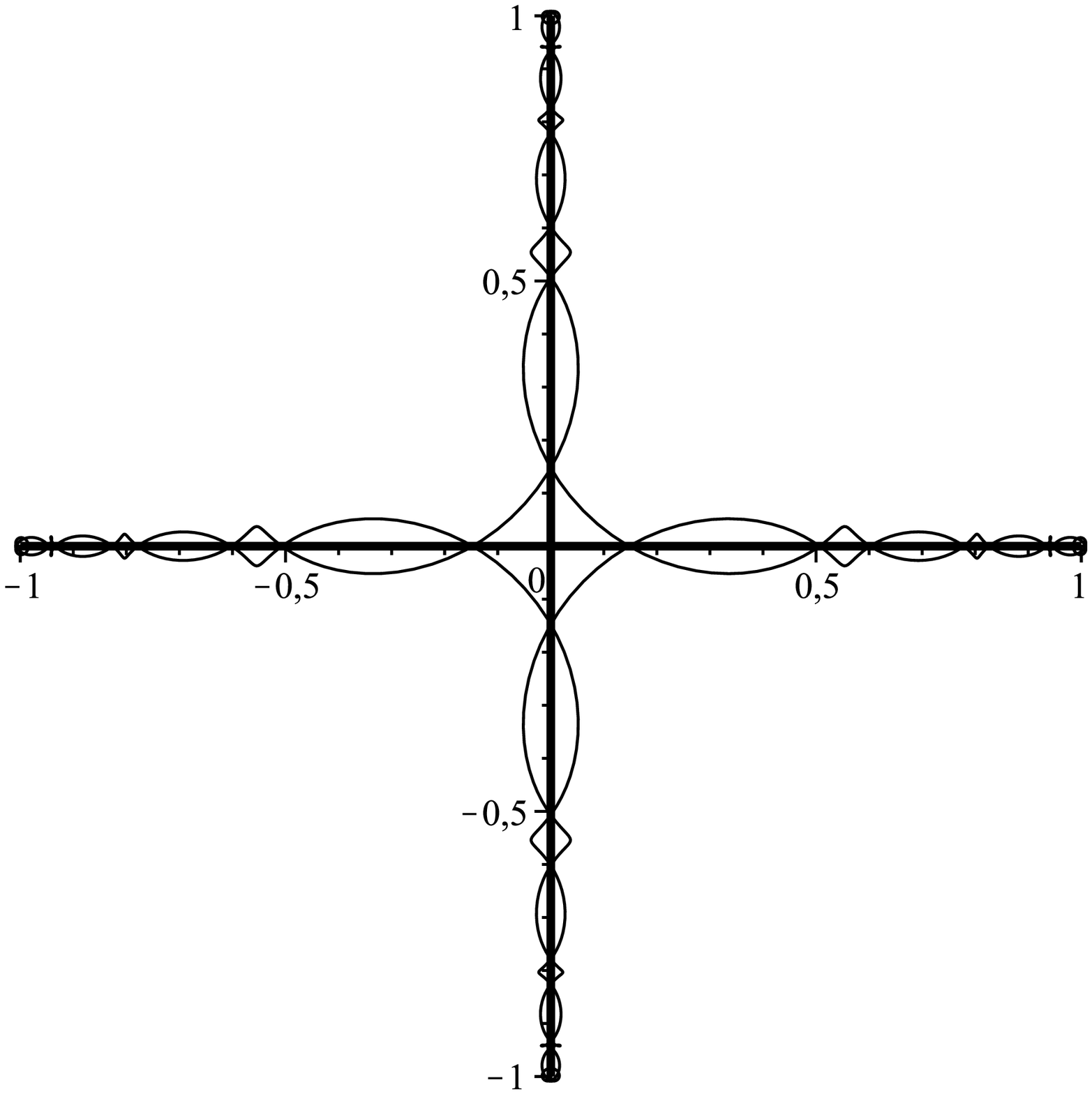}
\includegraphics[width=4.0cm,keepaspectratio=true]{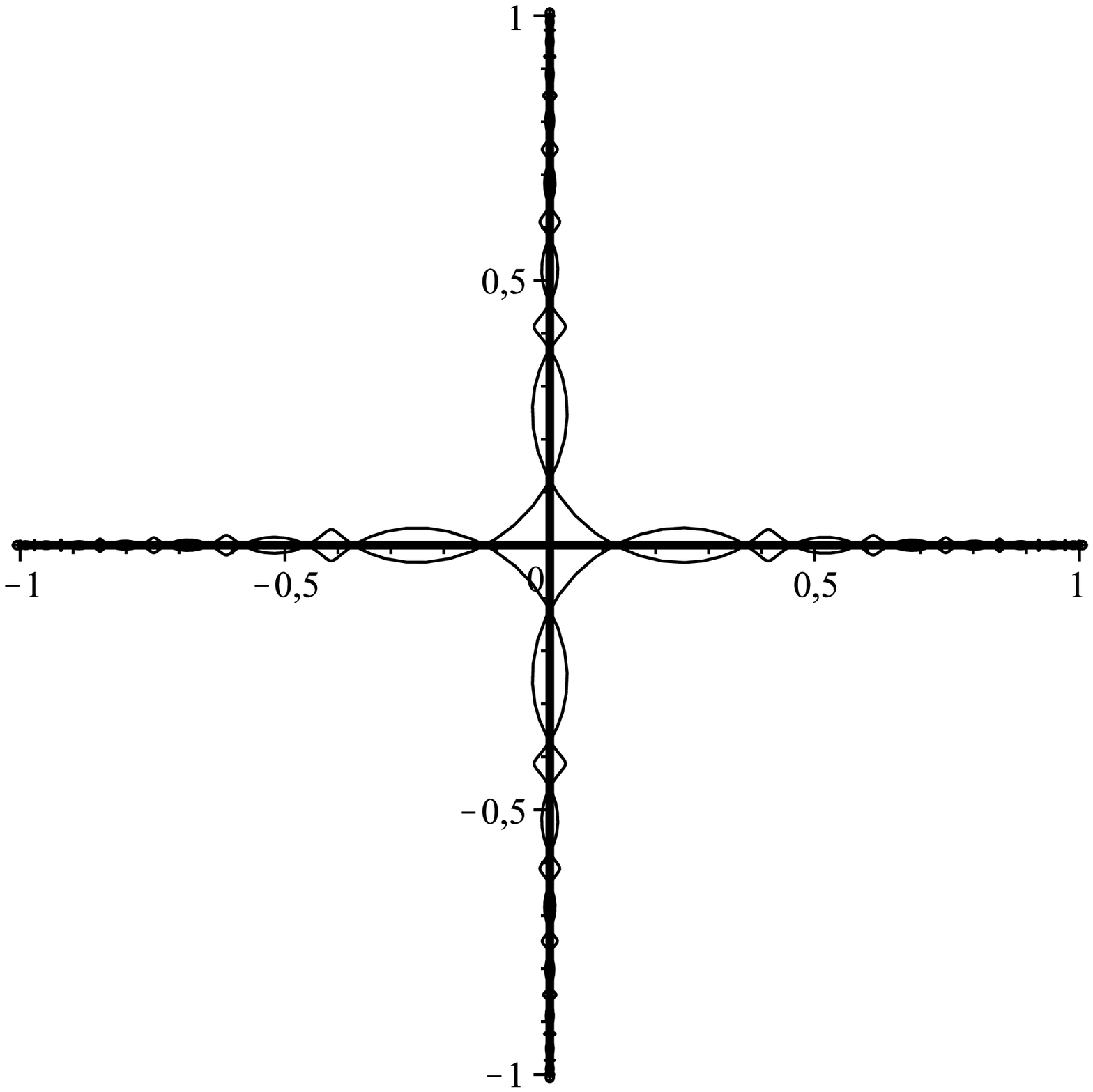}
\caption{$h_n (\TT )$ for $n=12$, $n=32$,
and $n=60$, respectively}
\end{center}
\end{figure}

\begin{figure}[H]
\begin{center}
\includegraphics[width=4.0cm,keepaspectratio=true]{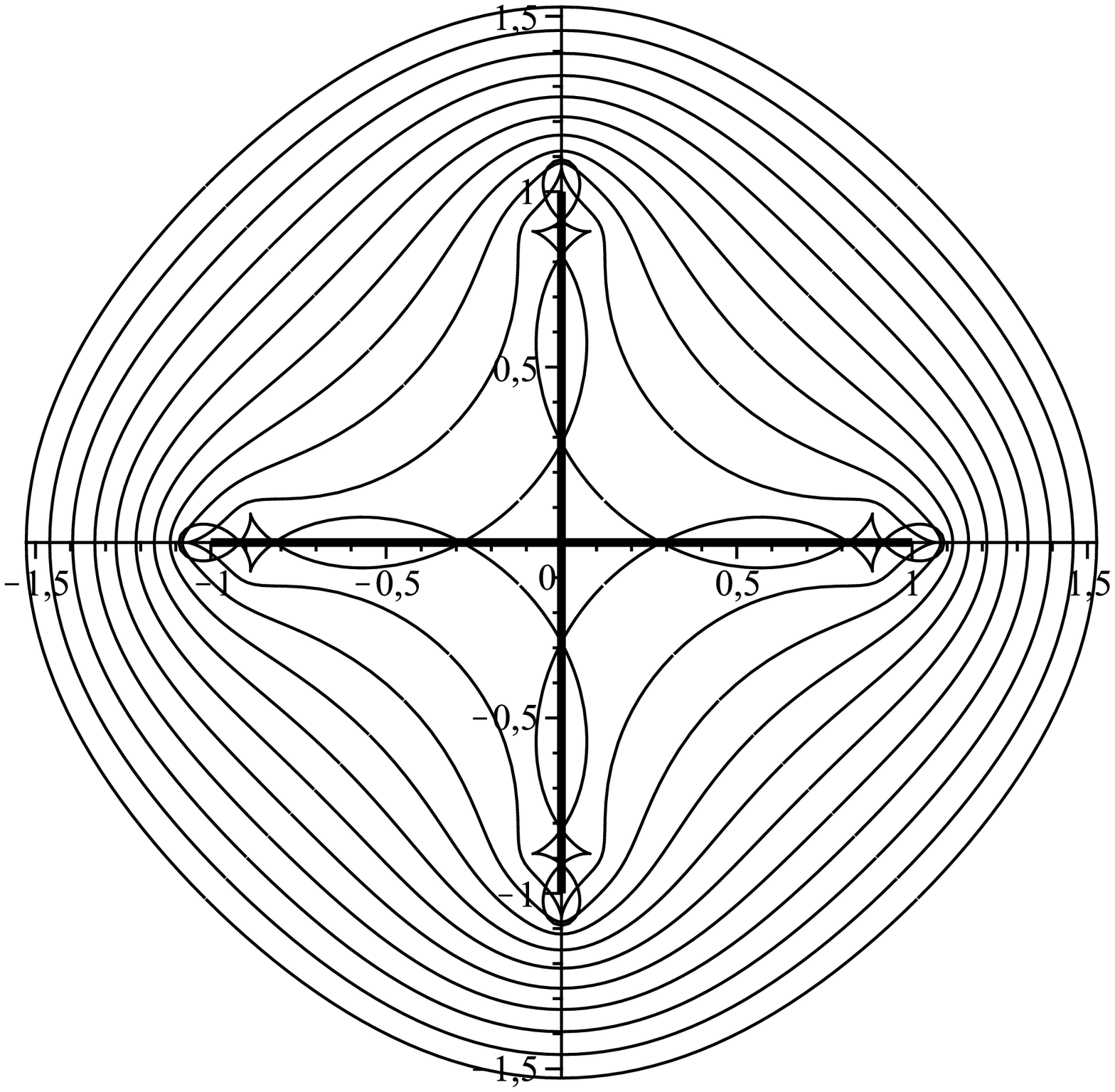}
\includegraphics[width=4.0cm,keepaspectratio=true]{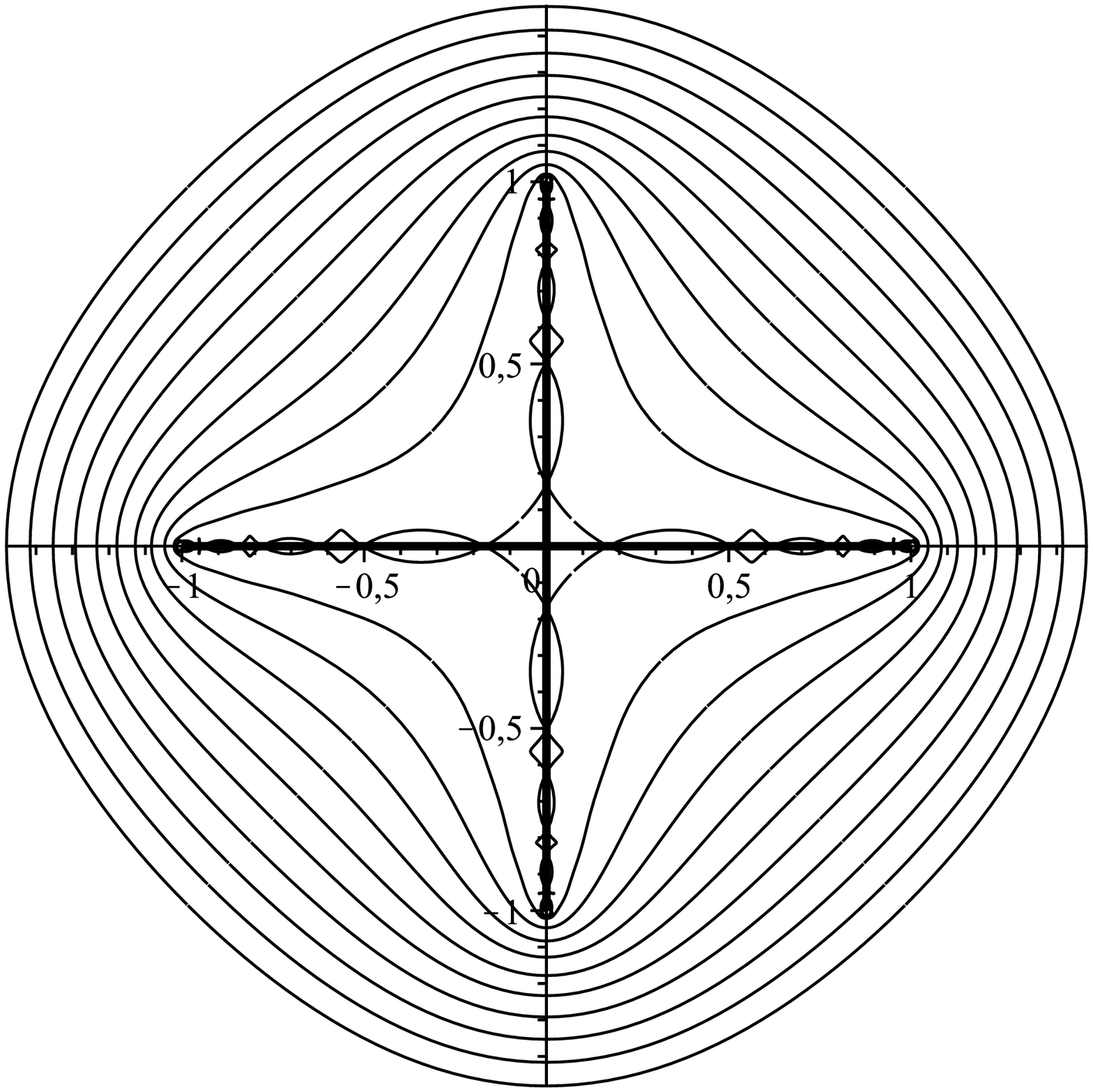}
\includegraphics[width=4.0cm,keepaspectratio=true]{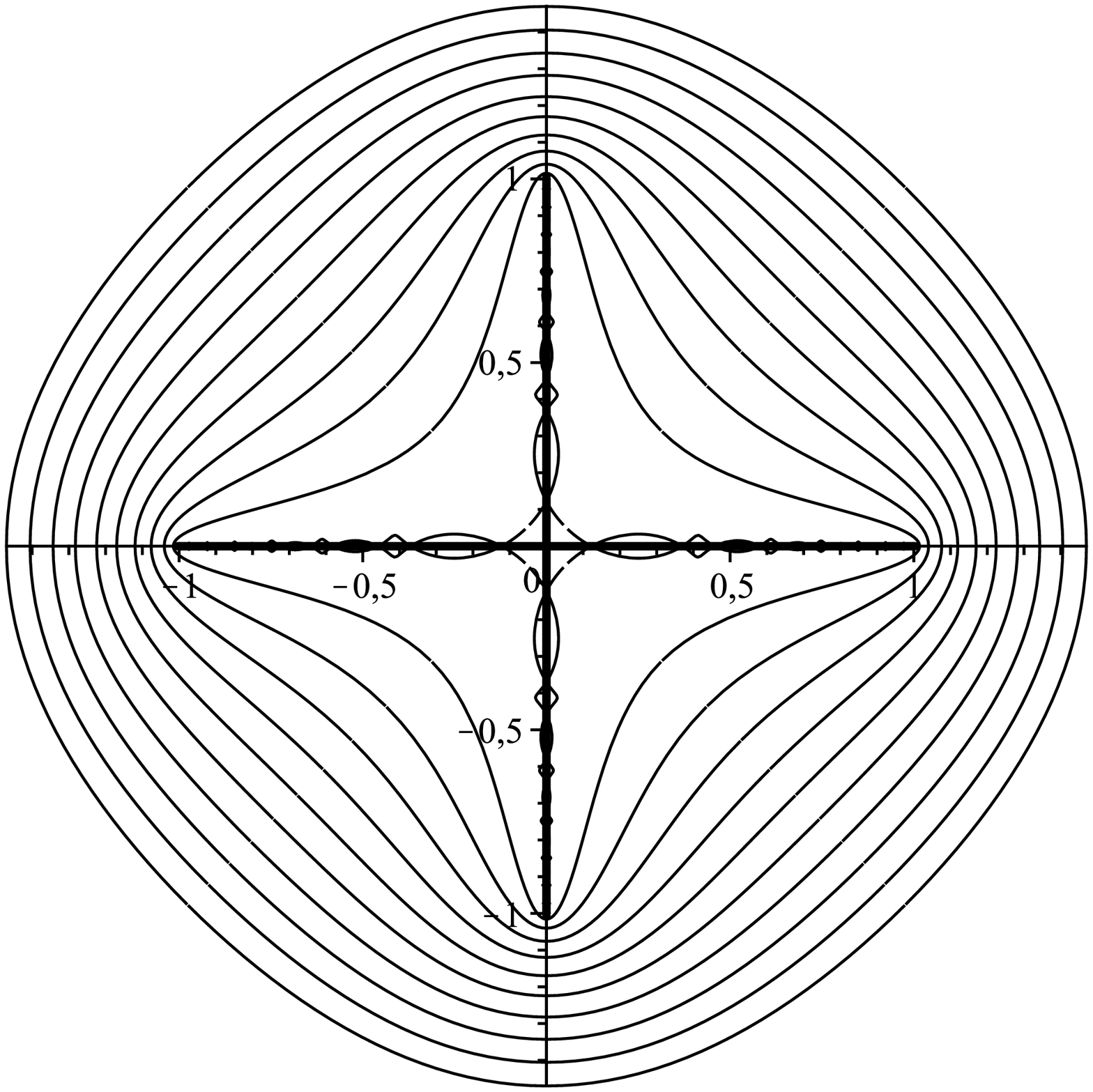}
\caption{Approximations of the Riemann mapping function using $h_n (re^{i\theta} )$, $n=12, 32, 60$, $r\in [1,2]$}
\end{center}
\end{figure}
As can be seen, the difference between $\phi (z) $ and $h_n (z) $ decreases as $|z|$ increases.
In the following figures we show some close-ups of the above figures.
\begin{figure}[H]
\begin{center}
\includegraphics[width=4.0cm,keepaspectratio=true]{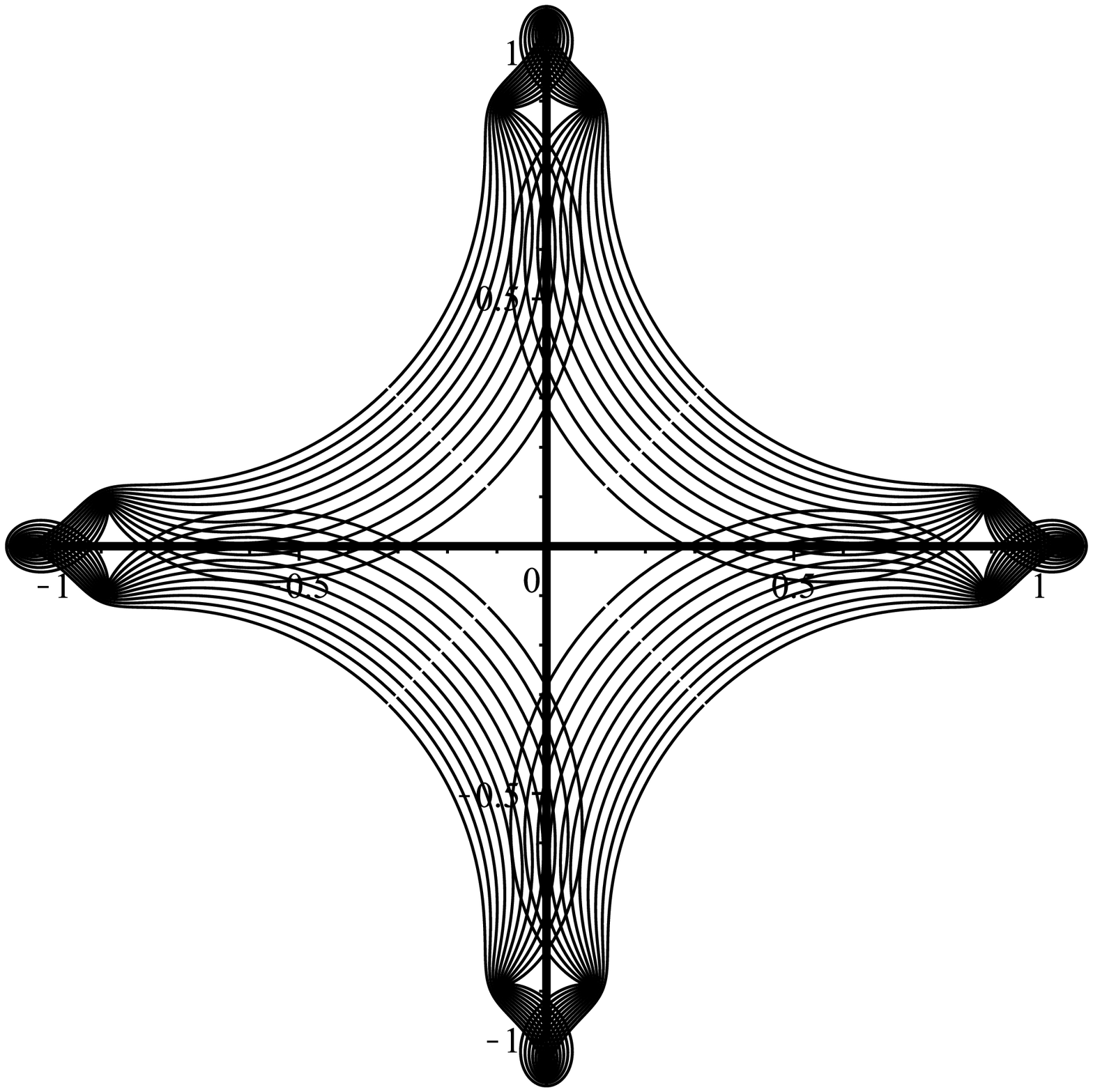}
\includegraphics[width=4.0cm,keepaspectratio=true]{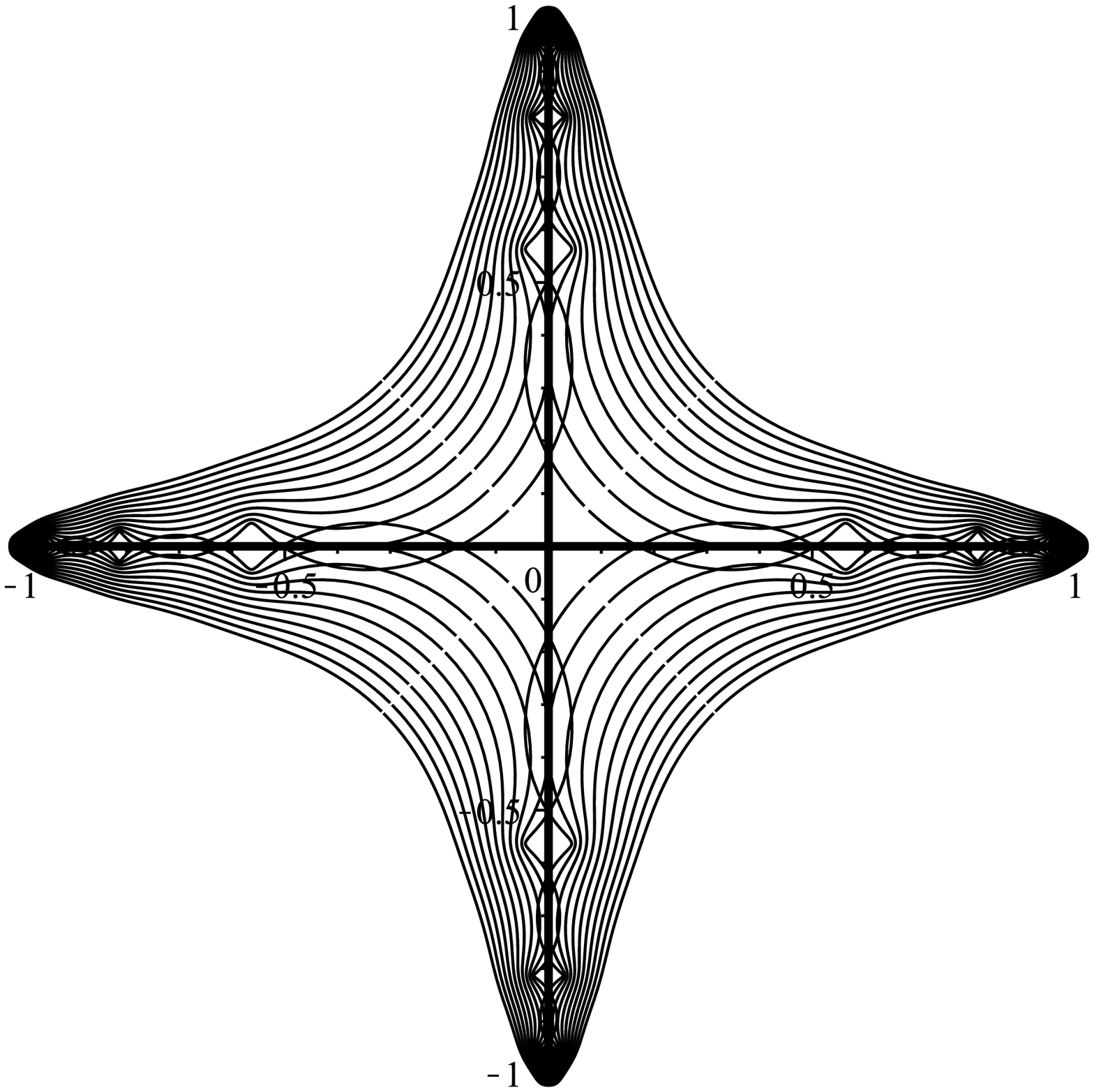}
\includegraphics[width=4.0cm,keepaspectratio=true]{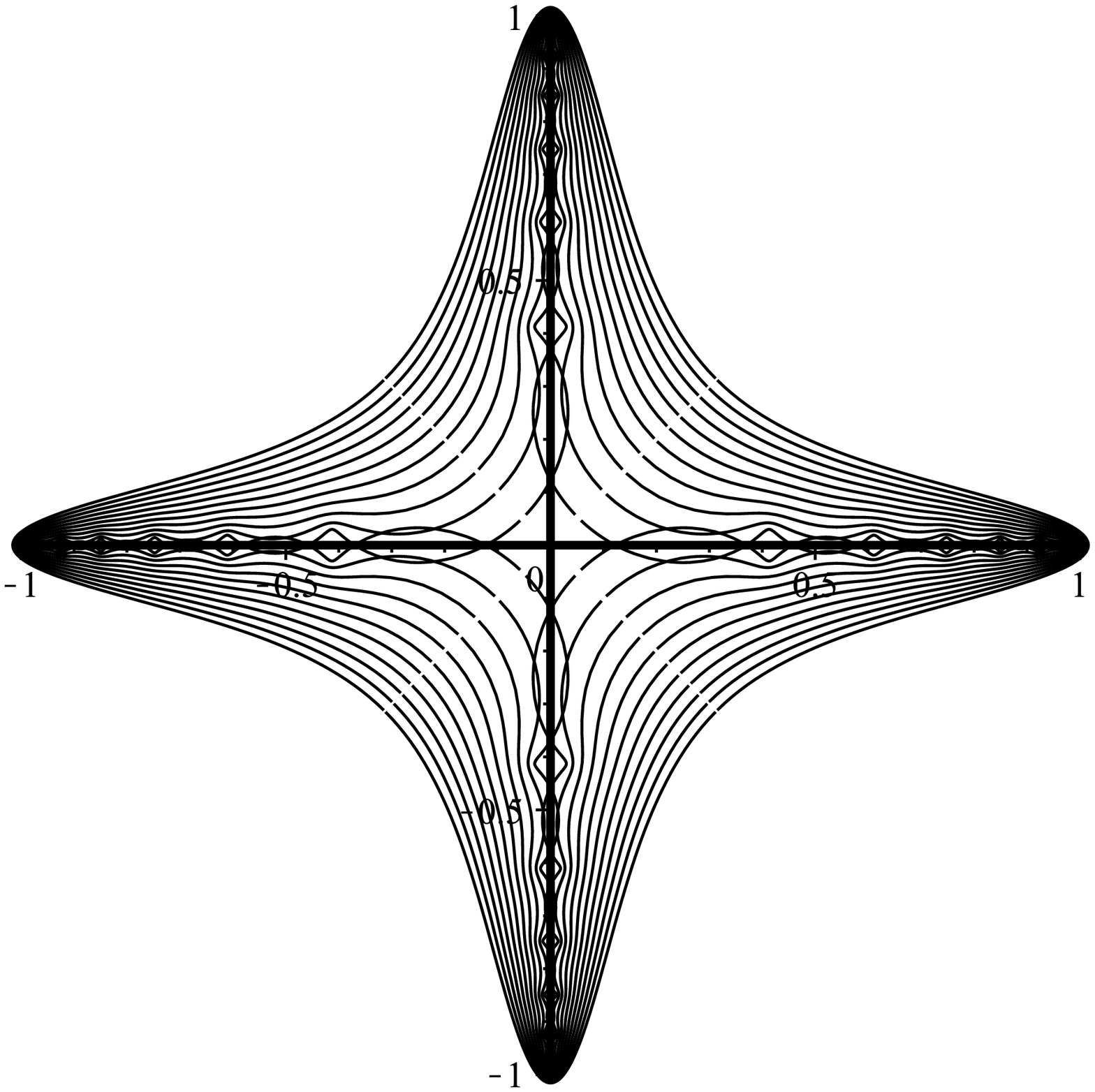}
\caption{Approximations of the Riemann mapping function using $h_n (re^{i\theta} )$, $n=12, 32, 60$, $r\in [1,1.1]$}
\end{center}
\end{figure}
\end{examp}

\begin{examp}
In the following example we take $\Gamma $ as the half part of a drop-like set of parametric equation
$$z(t)=\dfrac{(e^{i t})^2}{1+2 e^{i t}},t\in [0,\pi ],$$
and $\mu $ the uniform measure on $\Gamma $.

 Although in this case we do not know if the matrix $D-T$ defines a compact operator, the following figures seem to indicate that the convergence of $h_n(\TT)$ is very fast to $\Gamma$.  In the following figure we show several approximations of the support of $\mu $ using this method.
\begin{figure}[H]
\begin{center}
\includegraphics[width=4.75cm,keepaspectratio=true]{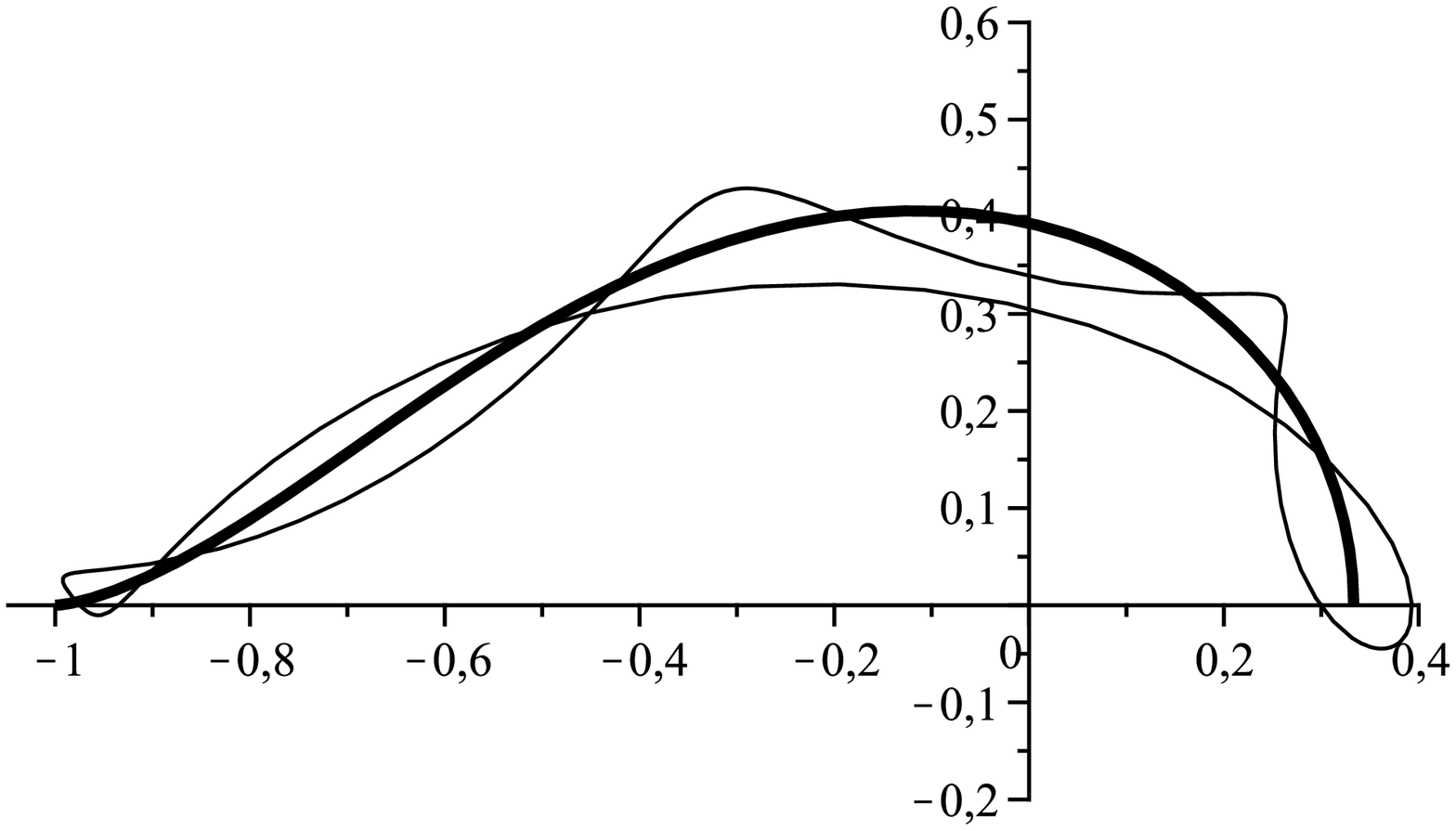}
\includegraphics[width=4.75cm,keepaspectratio=true]{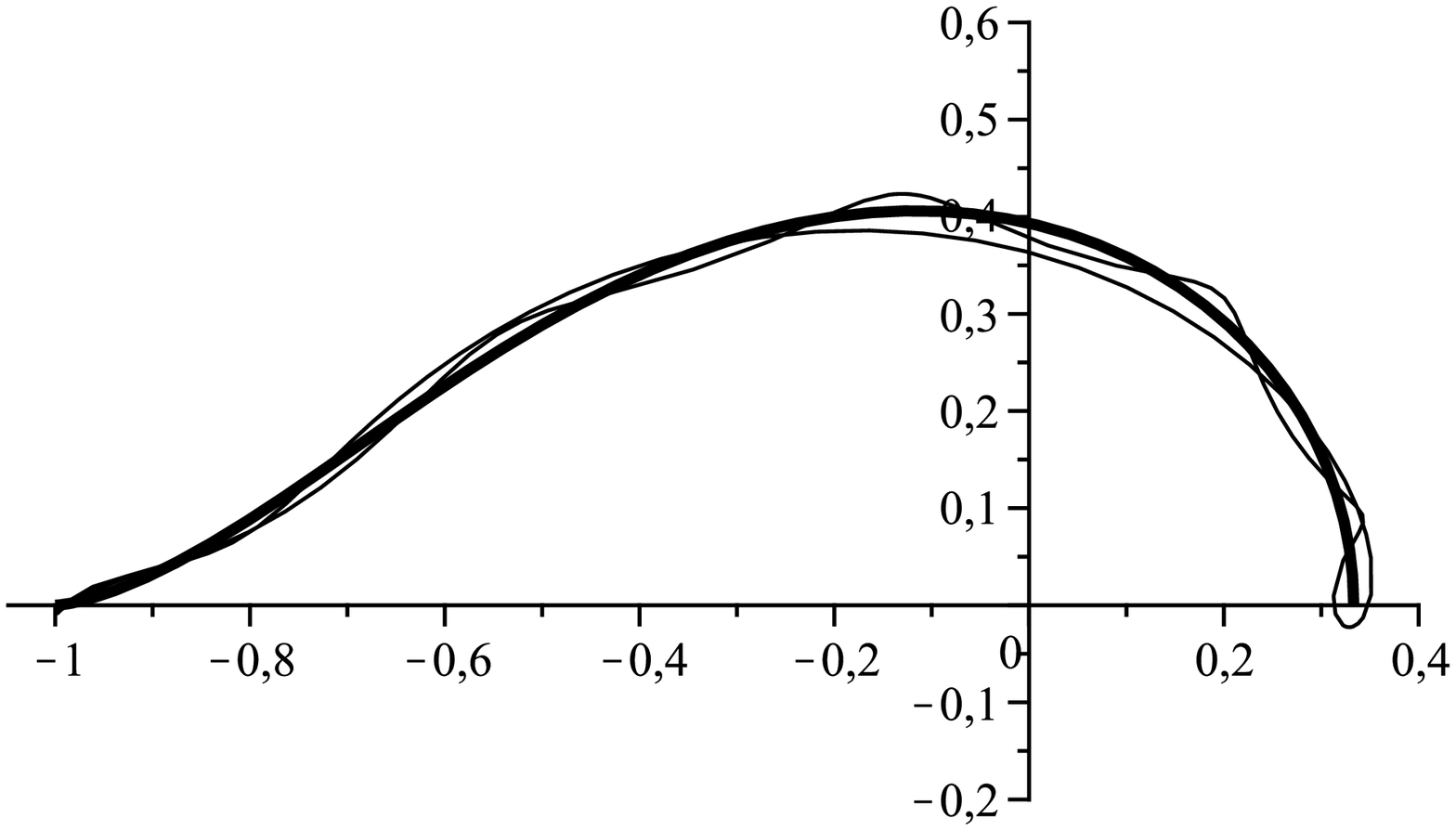}
\caption{$h_n (\TT )$ for $n=5$, $n=8$ and $n=11$, respectively}
\end{center}
\end{figure}
\end{examp}

\begin{examp}
For the last example we take $\Gamma $ as the spiral with parametric equation
$$z(t)=t \dfrac{e^{i t}}{6},t\in [0,2\pi ]$$
and we consider $\mu $ the uniform measure on $\Gamma $.

Although in this case we do not know if the matrix $D-T$ defines a compact operator, the following figures seem to indicate that the convergence of $h_n(\TT)$  to $\Gamma $ is worse than in the previous example. In the following figure we show several approximations of the support of $\mu $ using this method.
\begin{figure}[H]
\begin{center}
\includegraphics[width=4.75cm,keepaspectratio=true]{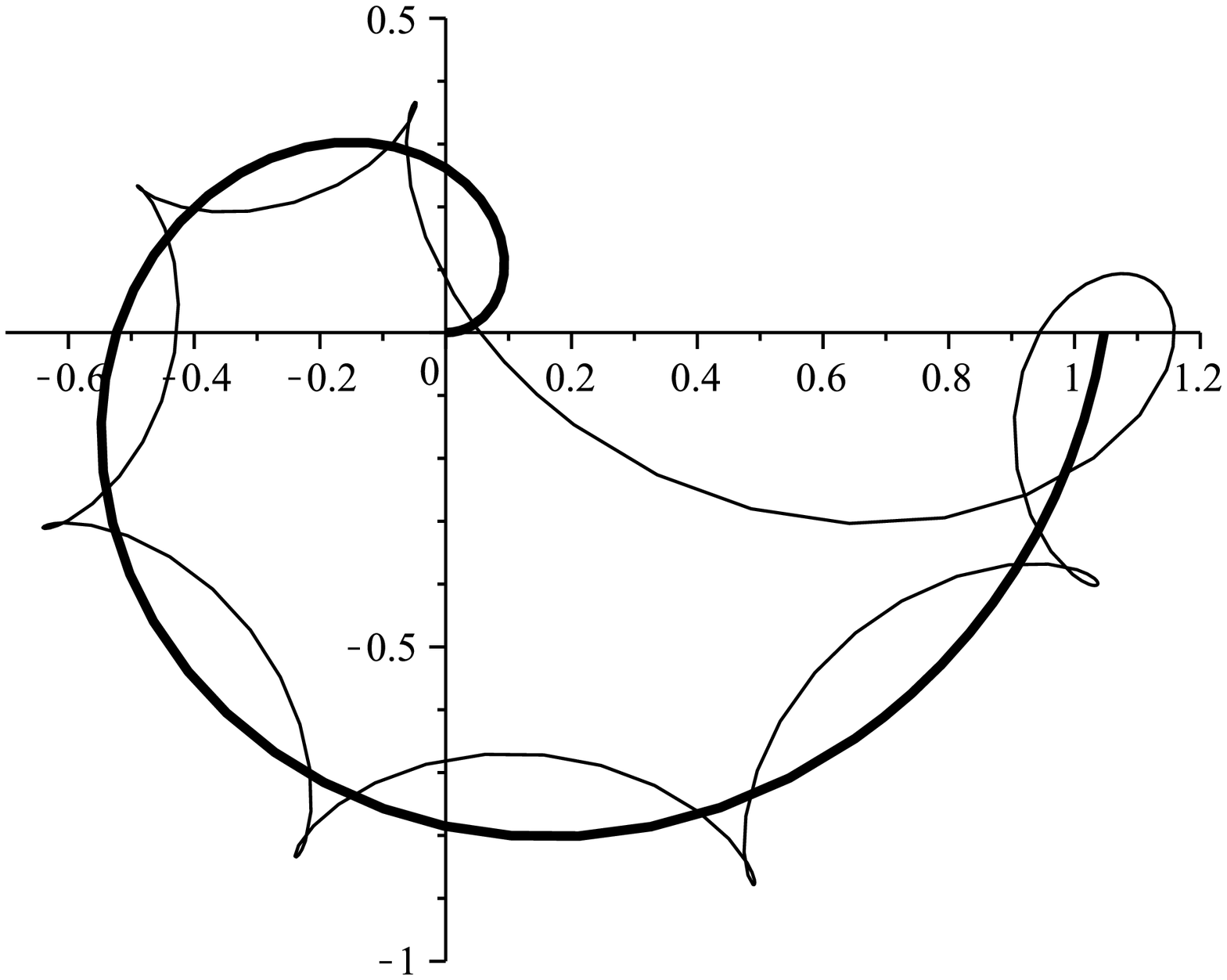}
\includegraphics[width=4.75cm,keepaspectratio=true]{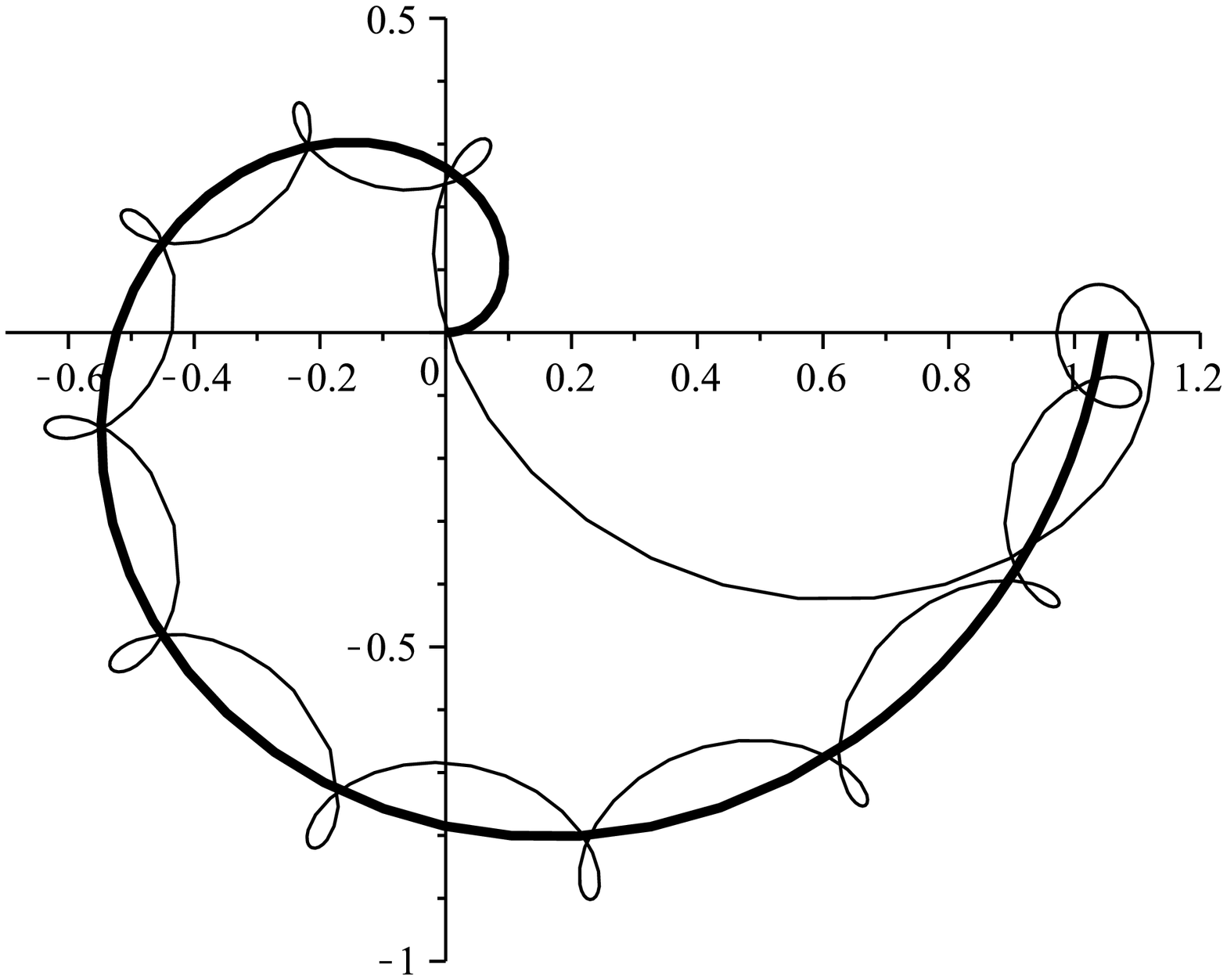}
\caption{$h_n (\TT )$ for $n=7$ and $n=11$, respectively}
\end{center}
\end{figure}
\end{examp}

\section*{Acknowledgements}

The authors have been supported by Comunidad Aut\'{o}\-no\-ma de Madrid and Universidad Polit\'{e}cnica de Madrid (UPM-CAM Q061010133).

\end{document}